\documentclass[11pt,a4paper,reqno]{amsart}
\usepackage[includehead,includefoot,margin=25mm]{geometry}
\usepackage{times}
\usepackage{amsfonts} %
\usepackage{amsmath} %
\usepackage{amssymb} %
\usepackage{amsthm} %
\usepackage{enumerate} %
\usepackage[all]{xy}
\usepackage{bbm} %
\usepackage{graphicx} %
\usepackage{nicefrac} %
\usepackage{color} %
\usepackage{hyperref}
\usepackage[utf8]{inputenc}
\hypersetup{
    colorlinks=true,       
    linkcolor=blue,          
    citecolor=blue,        
    filecolor=blue,      
    urlcolor=blue           
}

\usepackage{bm}

\newtheorem{theorem}{Theorem}[section] %
\newtheorem{corollary}[theorem]{Corollary} %
\newtheorem{lemma}[theorem]{Lemma} %
\newtheorem{proposition}[theorem]{Proposition} %

{\theoremstyle{remark} %
  \newtheorem{remark}[theorem]{Remark}} %
{\theoremstyle{definition} %
}

\newcommand{\RR}[0]{\ensuremath{\mathbb{R}}}

\newcommand{\ep}{\varepsilon}
\newcommand{\io}{\iota}

\newcommand{\scal}[0]{\ensuremath{\operatorname{Scal}}}
\newcommand{\ricc}[0]{\ensuremath{\operatorname{Ric}}}
\newcommand{\diver}[0]{\ensuremath{\operatorname{div}}}
\newcommand{\tr}{\ensuremath{\operatorname{tr_\g}}}

\newcommand{\g}[0]{\ensuremath{\mathrm{g}}}
\newcommand{\geph}[0]{\ensuremath{g+\ep^2 h}}

\begin{document}

\title[]{Concentration Phenomena for Conformal Metrics with Constant $Q$-Curvature}

\author{Salom\'on Alarc\'on}
\address{Departamento de Matem\'atica, Universidad T\'ecnica
  Fe\-de\-ri\-co San\-ta Ma\-r\'\i a, Valpara\'\i
  so, Chile}  \email{salomon.alarcon@usm.cl}

  \author{Simón Masnú}
\address{Facultad de Matem\'aticas, Pontificia Universidad Católica 
  de Chile, Santiago, Chile}  \email{smasnub@estudiante.uc.cl}

\author{Pedro Montero}
\address{Departamento de Matem\'atica, Universidad T\'ecnica
  Fe\-de\-ri\-co San\-ta Ma\-r\'\i a, Valpara\'\i
  so, Chile}  \email{pedro.montero@usm.cl}

\author{Carolina Rey}
\address{Departamento de Matem\'atica, Universidad T\'ecnica
  Fe\-de\-ri\-co San\-ta Ma\-r\'\i a, Valpara\'\i
  so, Chile}  \email{carolina.reyr@usm.cl}


\thanks{
  \mbox{\hspace{11pt}}{\it Key words}: $Q$-curvature, spaces of constant curvature, bi-laplacian operator.\\
  \mbox{\hspace{11pt}} S. Alarcón is partially supported by Fondecyt Projects 1211766 and 1221365. S. Masnú is partially supported by ANID through Beca de Doctorado Nacional 21241432. P. Montero is partially supported by Fondecyt Projects 1231214 and 1240101. C. Rey is partially supported by Fondecyt Project 3200422.}

\begin{abstract}
Let $(M,g)$ be an analytic Riemannian manifold of dimension $n \geq 5$. 
In this paper, we consider the following constant $Q$-curvature type equation
\begin{equation*}
\ep^4\Delta_{g}^2 u   -\ep^2 b \Delta_{g} u +a u = u^{p} , \qquad \text{in } M, \quad u>0, \quad u\in H^2_g(M)
\end{equation*}
where $a,b$ are positive constants such that $b^2-4 a>0$, $p$ is a sub-critical exponent $1<p<2^\#-1=\frac{n+4}{n-4}$,  $\Delta_g=\diver\nabla$ denotes the Laplace-Beltrami operator and $\Delta_g^2:=\Delta_{g}(\Delta_{g})$ is the bi-laplacian operator on $M$.

We show that if $\ep>0$ is small enough, then there are positive solutions to the above constant $Q$ curvature equation that concentrates around a maximum or minimum point of the function $\tau_g$, given by
\[
     \tau_g(\xi):= \sum_{i, j=1}^{n} \frac{\partial^{2} \g_{\xi}^{i i}}{\partial z_{j}^{2}}(0),
 \]
where $g_{\xi}^{i j}$ denotes the components of the inverse of the metric $g$ in normal geodesic coordinates. This result shows that the geometry of $M$ plays a crucial role in finding solutions to the equation above and provides a metric of constant $Q$-curvature on a product manifold of the form $(M\times X, g+\ep^2 h)$  where $(M,g)$ is Ricci-flat and closed, and $(X,h)$ any $m$-dimensional Einstein Riemannian manifold, $m> n+4$.
\end{abstract}

\maketitle

\tableofcontents

\section{Introduction}

\smallskip

\subsection{Motivation}

The Yamabe problem, a central question in differential geometry, asks whether a closed Riemannian manifold can admit a conformal metric with constant scalar curvature. This problem opened up numerous discussions in partial differential equations and differential geometry, pushing forward research in both fields. At the heart of this study is the notion of conformal invariance of differential operators.

A major milestone in this field was the complete resolution of the classical Yamabe problem. Through the efforts of H. Yamabe, Trudinger, Aubin, and Schoen (\cite{yamabe1960deformation, trudinger1968remarks, aubin1976equations, schoen1984conformal}) it was proved that conformal metrics with constant scalar curvature do exist in any closed Riemannian manifold.  This success established the foundation for addressing more advanced problems, such as the fractional Yamabe problem, which is a natural extension of the classical version. 

\medskip

The fractional Yamabe problem focuses on finding a conformal metric on a Riemannian manifold $(\mathcal{M}, \g)$ with constant fractional scalar curvature $\mathcal{Q}^\gamma_{\g}$. This scalar curvature is defined as 
$$\mathcal{Q}^{\gamma}_{\g} = \mathcal{P}^{\gamma}_{\g}(1),$$
where $\mathcal{P}^{\gamma}_{\g}$ is the fractional conformal Laplacian,  which exhibits the important property of conformal covariance, analogous to the classical conformal Laplacian. This problem, introduced in \cite{gonzalez2013fractional,gonzalez2018further}, has sparked extensive research on the existence of such conformal metrics and has also been fully resolved, with notable contributions from M. González, and J. Qing in \cite{gonzalez2013fractional}, M. González and M. Wang in \cite{gonzalez2018further}, S. Kim, M. Musso, and J. Wei in \cite{kim2017existence}, M. Mayer, C. Ndiaye in \cite{mayer2017fractional} and C. Ndiaye, and Y. Sire, and L. Sun in \cite{ndiaye2021uniformization}.

Specifically, the conformal property implies that under a conformal change of metric $\g_w = w^{\frac{4}{n-2\gamma}} \g$, the fractional conformal Laplacian transforms as:
\begin{equation}\label{Conformal-invariance}
\mathcal{P}^{\gamma}_{{\g}_w}(u) = w^{-\frac{n+2\gamma}{n-2\gamma}} \mathcal{P}^{\gamma}_{\g}(wu).
\end{equation}
For $\gamma = 1$, the operator reduces to the classical conformal Laplacian, and the fractional scalar curvature $\mathcal{Q}^{1}_{\g}$ becomes a constant multiple of the scalar curvature $\scal_{\g}$. When $\gamma = 2$, the operator corresponds to the Paneitz operator $P_{\g}$, and $\mathcal{Q}^{2}_{\g}$ becomes Branson's $Q$-curvature (\cite{branson1985differential, Paneitz}), directly linking the fractional Yamabe problem to the $Q$-curvature problem. 
The $Q$-curvature has rapidly established itself as a distinct area of investigation and has been the subject of extensive research. In the context of 4-dimensional manifolds, the $Q$-curvature assumes a critical role, given its relation to the total $Q$-curvature via the Chern-Gauss-Bonnet formula. Although the case $n=4$ has garnered considerable attention due to its direct association with topological invariants, the examination of $Q$-curvature in higher dimensions remains equally significant, as it extends fundamental geometric principles to a wider class of manifolds.

\medskip

In this paper, we focus on the constant $Q$-curvature problem for dimensions $n \geq 5$. 
On a $n-$dimensional Riemannian manifold $(\mathcal{M}, \g)$, the $Q$-curvature is expressed as
\[
	Q_{\g}=-\frac{1}{2(n-1)} \Delta_{\g} \scal_{\g}-\frac{2}{(n-2)^{2}}\left\|\mathrm{Ric}_{\g}\right\|^{2}+\frac{n^{3}-4 n^{2}+16 n-16}{8(n-1)^{2}(n-2)^{2}} \scal_{\g}^{2},
\]
where $\Delta_{\g} u=\diver_{\g}(\nabla_{\g} u)$ is the Laplace-Beltrami operator on $(\mathcal{M}, \g)$ and $\left\|A\right\|^2=\tr(AA^t)$, while $\mathrm{Ric}_{\g}$ and $\scal_{\g}$ represent the Ricci tensor and scalar curvature, respectively.

Similarly to the Yamabe problem, a key question is whether a conformal metric to $\g$ can be found with constant $Q$-curvature, which reduces to solving a fourth-order elliptic equation. In order to introduce this equation, we need to define the Paneitz operator $P_{\g}$ that satisfies the critical conformal invariance property \eqref{Conformal-invariance} with $\gamma=2.$
Therefore, the constant $Q$-curvature equation for the metric $\g=u^{\frac{4}{n-4}} \mathrm{\g}_{0}$ reads
\begin{equation}\label{ConstantQcurv}
	P_{\g_{0}} u=\lambda u^{\frac{n+4}{n-4}}, \quad \lambda \in \mathbb{R}.
\end{equation}
This equation, far from being a simple extension of the Yamabe problem, raises new questions about curvature in higher dimensions. When considering a local g-orthonormal frame $\left(e_{i}\right)_{i=1}^{n}$, the Paneitz operator can be expressed as 
$$
P_{\g} \psi=\Delta_{\g}^{2} \psi+\frac{4}{n-2} \diver_{\g}\left(\ricc_{\g}\left(\nabla \psi, e_{i}\right) e_{i}\right)-\frac{n^{2}-4 n+8}{2(n-1)(n-2)} \diver_{\g}\left(\scal_{\g} \nabla \psi\right)+\frac{n-4}{2} Q_{\g} \psi .
$$

The constant $Q$-curvature problem is also well defined when $\mathcal{M}$ is $3$ or $4$-dimensional, but it has a different expression. Many authors studied the four-dimensional equation, for example, A. Chang and P. C. Yang in \cite{Chang}, S. Brendle in \cite{Brendle},  Z. Djadli and A. Malchiodi in \cite{Djadli}, J. Li, Y. Li, and P. Liu in \cite{LiLiLiu}, among others.  In dimension $3$ some existence theorems were proved in \cite{HangYang16}.

We are interested in studying the existence of positive solutions to the equation \eqref{ConstantQcurv} on manifolds of dimension larger than $5$. 
In \cite{Malchiodi}, Gursky and Malchiodi made key advancements by showing that under certain conditions ($Q_{\g}$ is non-negative and positive at some point in $\mathcal{M}$) there exist solutions to the constant $Q$-curvature equation. These results were generalized by Hang and Yang in \cite{Hang}, who developed an existence theory for manifolds of dimension $n \geq 5$ with semi-positive $Q$-curvature.  On the other hand, J. Qing and D. Raske established a positive solution of the Paneitz-Branson equation in \cite{Qing} on a locally conformally flat manifold with positive scalar curvature and a Poincar\'e exponent smaller than $\frac{N-4}{2}$. See also \cite{Hebey, Djadli, Esposito, Gursky, Robert, Yang} for further existence results.  

In the case of the round sphere, C. S. Lin \cite{Lin} used the moving-plane approach to classify conformal metrics with constant $Q$-curvature and found an explicit multidimensional family of conformal metrics with constant $Q$-curvature. 

In addition, E. Hebey and F. Robert in \cite{Hebey2}  proved that the Paneitz-Branson  equation is compact when the Paneitz operator is of strong positive type. There are other significant results on the compactness of the solution space, such as \cite{ Li, YanYanLi, Malchiodi2}. On the other hand, J. Wei and C. Zhao in \cite{Wei} built examples of non-compactness of the space of positive solutions in high dimensions. Recently, V'{e}tois \cite{Vetois} showed that if 
 $\g$ is Einstein and not isometric to the constant curvature metric on the sphere, it is the unique metric with constant  $Q$-curvature, up to scaling. This result mirrors Obata's theorem in the Yamabe problem (\cite{Obata}). 
 Moreover, in certain cases, multiple conformal metrics with constant $Q$-curvature exist, see for example \cite{alarcon2023multiplicity, bettiol2021nonuniqueness, julio2024global}.

\medskip

In this paper, we study \eqref{ConstantQcurv} on a Riemannian product, as in \cite{alarcon2023multiplicity}, building a smooth positive solution that concentrates on a certain point of the first factor. Concretely, we prove that if $(M,g)$ and $(X,h)$ are closed Riemannian manifolds, where $(M,g)$ is analytic and Ricci-flat, and $(X,h)$ is an Einstein manifold with positive scalar curvature, then, for sufficiently small $\ep  > 0$, the $N$-dimensional Riemannian product $(M\times X, g+\ep ^2 h)$ has a conformal metric of the form $ u^{\frac{N}{N-4}}(g+\ep ^2 h)$, where $u: M\to \RR$ concentrates at an isolated minimum or maximum point of a suitable function $\tau_g\in C^\infty(M)$ (see Theorem \ref{mainthm}).

\medskip

We ask $(M, g)$ to be an analytic Riemannian manifold because we need to compute the Taylor’s series of $g_\xi$, the pullback of the metric $g$ in normal coordinates. More precisely, if $(M,g)$ is a $C^\infty$ Riemannian manifold, we can compute the Taylor expansion of $g_\xi$, although it may not converge. 

\subsection{Setting of the problem and main result}
Let $(M,g)$ be any closed analytic n-dimensional Riemannian manifold which is  Ricci-flat (ie, $\ricc_g = 0 = \scal_g$), and $(X,h)$ be an m-dimensional Riemannian Einstein manifold, with $\ricc_h=\Lambda_0>0$, scalar curvature $\scal_h=m\Lambda_0$ and $m>n+4$.
We will be interested in positive solutions of the constant $Q$-curvature equation \eqref{ConstantQcurv} for the
product manifold $(M\times X, g+\ep^2 h)$, which is
\begin{equation}\label{lambda}
P_{\geph} u=\lambda_\ep u^{\frac{n+m+4}{n+m-4}}, \quad \lambda_\ep \in \RR.
\end{equation}
Note that $p:=\frac{n+m+4}{n+m-4}<\frac{n+4}{n-4}$. So, when studying solutions which depend only on one of the variables, the equation is sub-critical. Therefore, assuming that $u:M\rightarrow \RR$, we obtain the following expression for the constant $Q$-curvature equation (for details, see \cite{alarcon2023multiplicity}).
	\begin{equation}\label{MainEq2}
		\ep^4 \Delta^2_g u   - \ep^2 b \ \Delta_{g} u +a \ u= u^{p} , \qquad \text{in } M, 
	\end{equation}
where $1<p<2^\sharp-1=\frac{n+4}{n-4}$, and the constant coefficients are
	\begin{align}
	   a&=a_{n,m}=\frac{\Lambda_0^2m \ (N-4)}{2(N-2)^{2}}\left[-2+\frac{(N^{3}-4 N^{2}+16 N-16)}{8(N-1)^{2}}m\right]\label{constante_a}\\ 
    b&=b_{n,m}=\frac{N^{2}-4 N+8}{4(N-1)(N-2)}m\Lambda_0,\label{constante_b}
	\end{align}
  where $N=m+n$.

\begin{remark}
    Recall from \cite[Lemmas 2.2 and 2.3]{alarcon2023multiplicity} that the constants in \eqref{lambda}, \eqref{constante_a} and \eqref{constante_b} satisfy the next conditions.
    \begin{itemize}
        \item If $m=2$, then $\lambda_\ep<0$ for $N=5, 6, 7,8$ and $\lambda_\ep>0$ for $N\geq 9$. If $m\geq 3$, then $\lambda_\ep>0$.
        \item $a_{N,m} >0$ if $m\geq 3$ or $m=2$ and $N\geq 9$. 
        \item If $m\geq 3$, or $m=2$ and $N\geq 9$, then	$b_{n,m}>2\sqrt{a_{n,m}}$.
    \end{itemize}
From now on, we will assume that the dimension of the base manifold $(M,g)$ is 
 $n\geq 5$ and the dimension of the fiber $(X,h)$ is $m>n+4$.
\end{remark}
The compactness of the embedding $H^2(M) \hookrightarrow L^{p+1}(M)$ ensures that 
\[
\inf_{u\in C^\infty(M)\setminus\{0\}}\dfrac{\int_{M} u\,  P_{g+\ep^2 h} (u) \ dv_g }{\left( \int_{M} |u|^{p+1} dv_g \right)^{\frac{2}{p+1}}  }
\]
is achieved, and so problem \eqref{MainEq2} always has a solution for any $p \in (1, 2^\sharp -1)$. Moreover, in \cite{alarcon2023multiplicity}, the authors showed that there are at least $cat(M) + 1$ non-trivial positive solutions to equation \eqref{MainEq2} provided $\ep$ is small enough. Here $cat(M)$ denotes the Lusternik–Schnirelmann category of $M$, which means that the amount of solutions to \eqref{MainEq2} depends on the topological properties of the manifold $M$. 

Just as the topology of $M$ gives us a lower bound for the number of solutions to the equation, we will show that its geometry allows us to find solutions that are concentrated at a geometrically important point of the manifold. More precisely, in this paper we show that, if  $\ep>0$ is small enough, then  a positive solution to  the problem \eqref{MainEq2} is generated by an isolated minimum or maximum point of the function $\tau_g$, which is strongly related to the geometry of the manifold.

It would be interesting to study whether these results remain valid more generally by considering a $C^1$-stable critical point of $\tau_g$. In such a case, does a solution exist that concentrates around this point, as M. Micheletti and A. Pistoia demonstrated in \cite{micheletti2009role} for a Yamabe-type equation? The technical challenge would lie in obtaining more precise estimates for the higher-order terms of the energy functional associated with the equation.
%

In this paper, we prove a concentration result for equation \eqref{MainEq2} with $\ep>0$ small enough. In consequence, we will obtain for each $\ep>0$ small enough, a metric in $(M\times X)$ conformal to $\g+\ep h$ with constant $Q$-curvature for any Einstein Riemannian manifold $(X, h)$ of positive scalar curvature.  To achieve this, we applied the well-known  Lyapunov-Schmidt reduction method, which was introduced in \cite{floer,bahri, bahri1995} and has been used in many articles; for more details see the survey \cite{delpino2016}. Here we follow the approach employed by \cite{micheletti2009role,deng2011blow,rey2021multipeak} in the context of Riemannian manifolds.
 We now briefly describe this method and state the results we have obtained.

Let $H^{2}_\ep$ be the Hilbert space $H^2_g(M)$ equipped with 
 the norm
\begin{equation}\label{norm}
\Vert u\Vert_{\ep}^{2}:=\frac{1}{\ep^{n}}\Big(\ep^{4}\int_{M}\Delta^2_{g}u \ dV_g + b \ep^{2}\int_{M}|\nabla u|^{2} \ dV_g+\int_{M}a \ u^{2} \ dV_g\Big).
\end{equation}
We also denote by $L_{g,\ep}^{q}$ the Banach space $L_{g}^{q}(M)$ furnished with the norm
\begin{equation}\label{def:normLp}
|u|_{q,\varepsilon}:=\Big(\frac{1}{\ep^{n}}\int_{M}|u|^{q} \ dV_g\Big)^{1/q}.
\end{equation}

We will split the space $H^{2}_\ep$ into two subspaces, one finite-dimensional and its orthogonal complement. 
First, we introduce the embedding $\io_{\ep} :H^{2}_{\ep}\hookrightarrow L_{g,\ep}^{p+1}$, in order to rewrite the problem (\ref{MainEq2}) as $ u=\io_{\ep}^{*}(u_{+}^{p})$ (details in Section \ref{sec:Background}).
Then, we introduce the following equation in the Euclidean space, which is {\it the limit equation} for our problem.
\begin{equation}\label{limit}
\Delta^2 u -b \Delta u +  a u  = u^p  \quad \text{in } \RR^n,
\end{equation}
with $a, b >0$, and for $2<p+1=2_N^{\sharp}<2^{\sharp}$, where $2_N^{\sharp}=  \frac{2N}{N-4} $. Solutions to equation \eqref{limit} have been extensively studied in \cite{bonheure2018orbitally}, where they are found as critical points of the energy functional $E: H^2 (\mathbb{R}^n ) \rightarrow \mathbb{R}$, given by
$$
E (u) = \frac{1}{2}  \int_{\mathbb{R}^n} \left(  | \Delta u |^2  +
 b  | \nabla  u |^2  +a u^2 \right) dz -  \frac{1}{p+1} \int_{\mathbb{R}^n} |u|^{p+1} \ dz.
 $$
They prove that if $b \geq 2 \sqrt{a}$, then the infimum 
$$
\mu= \inf_{u \in H^2 (\mathbb{R}^n )\setminus\{0\}} E(u)
$$ 
 is achieved by a positive function that is radially symmetric around a certain point and strictly radially decreasing. Note that this minimizer naturally satisfies \eqref{limit}.

  We say that a solution $U$ to \eqref{limit} is non-degenerate in $H^2(\RR^n)$ if for any
solution $v$ of the linearized equation
 \[
 \Delta^2 v -b \Delta v +  a v  = pU^{p-1}v  \quad \text{in } \RR^n,
 \]
there exists $\xi \in \RR^n$ such that $v(x) = \xi \cdot \nabla U(x)$. In other words, 
if $v$ is a linear combination of the functions
\[
 \Psi^{i}=\frac{\partial U}{\partial z_i}, \quad \text{with }i=1,\dots, n.
\] 
In \cite[Theorem 1.3]{bonheure2018orbitally}, the authors prove the above property for ground-state solutions to equation
\[  
\gamma\Delta^2 u -b \Delta u +  a u  = u^{p}  \quad \text{in } \RR^n, 
\]
under the hypothesis that $1 < p < \frac{4+n}{n}$ and $\gamma > 0$ is small enough. 
They conjecture that the property holds probably without the smallness
assumption on $\gamma$. We claim that this is true. Indeed, consider the equation
 \begin{equation} \label{eq:no-deg}
    \tilde\gamma^4 \Delta^2 w - b\tilde\gamma^2 \Delta w + a w = w^{p} \quad \text{in } \mathbb{R}^n.
\end{equation}
where \( a, b, \tilde\gamma > 0 \).
From the results in \cite{bonheure2018orbitally}, we know that for $\tilde\gamma$ small enough, there exists a ground-state solution $w$ to \eqref{eq:no-deg} that is non-degenerate.
Setting \( U(x) = w\left(\tilde\gamma x\right) \), we observe that the effects of this change of variables are given by
\begin{equation*}
 \Delta^2 U - b \Delta U + a U = U^{p} \quad \text{in } \mathbb{R}^n
\end{equation*}
and
\[
\|U\|_{L^2(\mathbb R^n)}^2 = \tilde\gamma^{-n} \|w\|_{L^2(\mathbb R^n)}^2,
\]
and thus the non-degeneracy of solutions of \eqref{limit} follows.
Moreover, they show that the function $w$ (and in consequence $U$) and its derivatives are exponentially decaying at infinity,
namely, there is a positive constant $C$ such that
\begin{equation}\label{decay}
\displaystyle U(x)<\frac{C}{\sqrt{b^2-4a}}\displaystyle e^{\left(\frac{\sqrt{b-\sqrt{b^2-4a}}}{\sqrt{2}}-\delta\right)|x|}
\end{equation}
for any $\delta>0.$ The same kind of argument also implies that each derivative of $U$ has an exponential decay. 
Note that if we write $U_{\ep}(x)=U(\frac{x}{\ep})$, then $U_{\ep}$ is a solution of 
\begin{equation*}
\ep^4 \Delta^2 u -b\ \ep^2 \Delta u +  a u  = u^p  \quad \text{in } \RR^n.
\end{equation*}
Moreover, in the case of a product manifold, we have \( p = \frac{n+m+4}{n+m-4} \). For the condition \( 1 < p <  \frac{4+n}{n} \) to hold, it is necessary to require \( m > n + 4 \).

\medskip

Our goal is to find an approximation for the solution to equation \eqref{MainEq2} by transforming the function $U$, which is defined in $\mathbb{R}^n$, into a function in $M$ that concentrates on a specific point.

Let $r$ be the injectivity radius of $M$, and let $\chi_r$ be a positive smooth cutoff function such that 
\[
\chi_r(z)=
     \begin{cases}
       1 &\quad\text{if } z \in B(0,r/2)\\
       0 &\quad\text{if } z \in \RR^n\backslash B(0,r),
     \end{cases}
\]
with bounded derivatives up to fourth order, i.e., there are positive constants $c_1,c_2,c_3,c_4$ such that
\begin{equation}\label{chi-decay}
    |\partial_i \chi_r(z)| \leq c_1, \quad|\partial_{ij}^2 \chi_r(z)| \leq c_2, \quad |\partial_{ijk}^3 \chi_r(z)|\leq c_3,\quad \textup{ and }\quad |\partial_{ijk\ell}^4 \chi_r(z)|\leq c_4.
\end{equation}
For any point $\xi \in M$ and for any positive real number $\lambda$, we define the {\it approximate} solution as the function $W_{\ep, \xi}$ on $M$ given by
    \begin{equation}\label{def:functionW}
        W_{\varepsilon,\xi}(x):= \left\{ \begin{array}{lcc}
U_{\ep}(\exp_{\xi}^{-1}(x))\chi_{r}(\exp_{\xi}^{-1}(x)) &   \hbox{ if }  &x\in B_g(\xi,r),  \\
 0 &   & \hbox{ otherwise},
        \end{array}\right.
    \end{equation}
    where $\exp_{\xi}:T_{\xi}M \rightarrow M$ is the geometrical exponential map, $B_g(\xi,r)$ denotes the ball in $M$ centered at $\xi$ with radius $r$ with respect to the distance induced by the metric $g$, and $v_\ep(x)=v(\frac{x}{\ep})$ for all $v\in H^2(\RR^n)$.
\smallskip
 For $i=1,\dots, n$ we denote $\Psi_{\ep}^{i}(z):= \Psi^{i}(\ep^{-1}z),$ and define on $M$ the functions 

\begin{equation}\label{def:functionZ}
Z_{\ep,\xi}^{i}(x):= \left\{ \begin{array}{lcc}
\Psi_{\ep}^{i}(\exp_{\xi}^{-1}(x))\chi_{r}(\exp_{\xi}^{-1}(x)) &   \hbox{ if }  &x\in B_g(\xi,r),  \\
 0 &   & \hbox{ otherwise}. 
\end{array}
\right.
\end{equation}
We will look for a solution $u\in H^2_g(M)$ to \eqref{MainEq2}, or equivalently to 
$$u- \io_{\ep}^{*}(u_{+}^{p}) =0$$ of the form
$
u= W_{\varepsilon,\xi}+\phi,
$
with $\|\phi\|_\ep\to 0.$
%
 So we are looking for a point $\xi\in M$ and a function $\phi$ such that
\begin{equation}\label{newMaineq}
W_{\varepsilon,\xi}+\phi- \io_{\ep}^{*}(f(W_{\varepsilon,\xi}+\phi)) =0.  
\end{equation}
To this end, for each $\ep >0$ and $\xi\in M$, we consider the spaces
\begin{align*}
     &  K_{\ep,\xi} :=\hbox{ span } \{Z_{\ep,\xi}^{i} : i=1, \dots, n\} \\
     & K_{\ep,\xi}^{\perp}:=\{\phi\in H^2_{\ep}: \left\langle \phi, Z_{\ep,\xi}^{i} \right\rangle_{\ep}=0, \text{ for all } i=1, \dots, n\},
\end{align*}
and the associated orthogonal projections $\Pi_{\ep,\xi}$ : $H^2_{\ep}\rightarrow K_{\ep,\xi}$ and $\Pi_{\ep,\xi}^{\perp}$: 
$H^{2}_{\ep}\rightarrow K_{\ep,\xi}^{\perp}$. 
Then, equation \eqref{newMaineq} is equivalent to the system
\begin{subequations}
\begin{align}
&\Pi_{\ep,\xi}^{\perp}\{  W_{\ep,  \xi}+\phi- \io_{\ep}^{*}(f(  W_{\ep,  \xi}+\phi)  \}=0, \label{Pii} \\
&\Pi_{\ep,\xi}\{  W_{\ep,  \xi}+\phi- \io_{\ep}^{*}(f(  W_{\ep,  \xi}+\phi) \}=0.\label{PiP}
\end{align}
\end{subequations}
The goal is to demonstrate, using the Inverse Function Theorem, that given $\ep>0$ small enough, and $\xi\in M$, there exists a unique small $\phi =\phi_{\ep,\xi} \in K_{\ep,\xi}^{\perp}$ that fulfills \eqref{Pii}. This is the essence of Proposition \ref{prop1}. Additionally, we need to find a point $\xi \in M$ that solves the finite-dimensional problem \eqref{PiP}. This is the main idea behind Proposition \ref{prop3}.

\vspace{2mm}

Our main result reads as follows.

 \begin{theorem}\label{mainthm} Let $(M,g)$ be a closed (compact boundaryless) analytic Riemannian manifold, and $\xi_0\in M$ be an isolated minimum or maximum point for the function $\tau_g:M\to \mathbb{R}$, which is given by
 \begin{equation}\label{tau}
     \tau_g(\xi):= \sum_{i, j=1}^{n} \frac{\partial^{2} g_{\xi}^{i i}}{\partial z_{j}^{2}}(0),
 \end{equation}
 where $g_{\xi}^{i j}$ denotes the components of the inverse of the metric $g$ in geodesic normal coordinates.
 Then, there exists $\varepsilon_0 > 0$ such that for any $\varepsilon\in (0, \varepsilon_0)$ problem \eqref{MainEq2} has a solution $u_{\varepsilon}$ which concentrates at $\xi_0$ as $\varepsilon$ goes to zero.
 \end{theorem}


 The Theorem \eqref{mainthm} allows us to find metrics on a product manifold with constant $Q$-curvature. So, as a Corollary, we have
\begin{corollary}
     Let $(M,g)$ be an  $n-$dimensional analytic Riemannian manifold, that is Ricci-flat and closed, and $(X,h)$ an  $m-$dimensional Einstein Riemannian manifold, with $n\geq 5$ and $m>n+4$. 
     If $\xi_0\in M$ is an isolated minimum or maximum point for the function $\tau_g$, then there exists $\varepsilon_0 > 0$ such that for any $\varepsilon\in (0, \varepsilon_0)$ there is a conformal metric to $g+\ep^2 h$ in the product manifold $M\times X$ with constant $Q$-curvature.
\end{corollary}

The structure of the paper is as follows. In Section \ref{sec:Background}, we provide the necessary notation and preliminary results. In Section \ref{sec:Finite}, we establish bounds for the involved operators and we show how to apply the Implicit Function Theorem to solve equation \eqref{Pii}. In Section \ref{sec:Expansion}, we examine the energy functional associated with \eqref{MainEq2} and illustrate its expansion around the approximate solution. In Section \ref{sec:Reduced}, we focus on solving the finite-dimensional problem \eqref{PiP} and we prove the Theorem \ref{mainthm}.
Finally, in Section \eqref{sec:Appendix} we compute some auxiliary estimations. 

\smallskip

\section{Background}\label{sec:Background}
In this section, we establish some notation and define the functional framework needed to study our problem. 

\smallskip

Let $H^{2}_\ep$ be the Hilbert space $H^2_g(M)$ equipped with the inner product 
\begin{equation*}
\left\langle u, v\right\rangle_{\ep}:=\frac{1}{\ep^{n}}\Big(\ep^{4}\int_{M}\Delta_{g}u\Delta_{g}v\ dV_g + b
\ep^{2}\int_{M}\nabla u\nabla v \ dV_g+\int_{M}a \ uv \ dV_g\Big),
\end{equation*}
which induces the norm defined in \eqref{norm}.
Similarly, on $\RR^n$ we define the inner product for  $u,v \in H^2(\RR^n )$ as
\[
\left\langle u, v\right\rangle:=\int_{\RR^n }\Delta u\Delta v \ dz + b\int_{\RR^n }\nabla u\nabla v \ dz +a\int_{\RR^n }  \ uv \ dz,
\]
which induces the norm
$
\Vert u\Vert:=\left\langle u, v\right\rangle.
$
We also denote by $L_{g,\ep}^{q}$ the Banach space $L_{g}^{q}(M)$ furnished with the norm
\eqref{def:normLp}.
 As usual, we identify the dual space $({L_{g,\ep}^{q}})^{'}$ with $L_{g,\ep}^{q'}$ with the pairing
\[
 \left\langle \varphi , \psi \right\rangle_{q,q'} = \frac{1}{\ep^{n}}\int_{M}\varphi \psi,
\]
\noindent
for $\varphi \in L_{g,\ep}^{q}$, $\psi \in L_{g,\ep}^{q'}$.
The Sobolev embedding theorem (see e.g. \cite{Hebey96}) asserts that $H_g^2(M)$ is continuously embedded in $L_g^q(M)$ for $1<q\leq 2^\#$, and this embedding is compact when $q<2^\#$. It follows that there exists a constant $c$
independent of $\ep$ such that

\begin{equation}\label{norma_q}
|u|_{q,\ep}\leq c\Vert u\Vert_{\ep} \quad \text{for any $u\in H^{2}_{\ep}$.}
\end{equation}

In consequence, the embedding $\io_{\ep} :H^{2}_{\ep}\hookrightarrow L_{g,\ep}^{p+1}$,
 and the adjoint operator $\io_{\ep}^{*}:L_{g,\ep}^{\frac{p+1}{p}}\rightarrow H^{2}_{\ep}$,  satisfy
$$
u=\io_{\ep}^{*}(v)\Leftrightarrow \left\langle \io_{\ep}^{*}(v),\displaystyle \varphi \right\rangle_{\ep}=\left\langle v, \io (\varphi) \right\rangle_{p+1,\frac{p+1}{p}}=\frac{1}{\ep^{n}}\int_{M}v\varphi, \quad  \text{for all } \varphi\in H^{2}_{\ep},
$$
which means
\begin{equation}
\label{Adjoint}
u=\io_{\ep}^{*}(v) \quad \Leftrightarrow \quad \ep^{4}\Delta^2_{g}u-\ep^{2}b\Delta_{g}u+a u=v \quad \hbox{ (weakly) on } M.
\end{equation}
Let 
$
f(u):=u_{+}^{p}.
$
Now, the problem (\ref{MainEq2}) is equivalent to
\begin{equation}\label{equivalent problem}
u=\io_{\ep}^{*}(f(u))\ ,\ u\in H^{2}_{\ep}.
\end{equation}
Moreover, from \eqref{Adjoint} and \eqref{norma_q},  we have the inequality
\begin{equation}\label{norma_ep}
\Vert \io_{\ep}^{*}(v)\Vert_{\ep}\leq c|v|_{\frac{p+1}{p},\ep} \quad \text{for any }
v\in L_{g,\ep}^{\frac{p+1}{p}} \quad \text{and some } c\in \RR \text{ independent of }\ep.
\end{equation}

Let us recall an important property of the exponential map $\exp_{\xi}: T_{\xi}M \rightarrow M$. There exists an injectivity radius $r > 0$ such that for any $\xi \in M$,  the restriction $$\exp_{\xi} |_{B(0,r)} : B(0,r) \rightarrow B_g(\xi,r)$$ is a diffeomorphism. Here, $B(0,r)$ denotes the ball in $\mathbb{R}^n$ centered at $0$ with radius $r$, and $B_g(\xi,r)$ denotes the ball in $M$ centered at $\xi$ with radius $r$ with respect to the distance 
$$
 d_g(x,y)=\exp_x^{-1}(y).
$$. 
This fact ensures us that the functions $W_{\ep,\xi}$ and $Z_{\ep,\xi}^i$, $i=1, \dots,n$, given by \eqref{def:functionW} and \eqref{def:functionZ}, are well defined. Moreover, it allows us to consider the pullback of the metric in $M$. 

\smallskip

Let $g_{\xi}$ denote the Riemannian metric read in $B(0, r)$ through the normal coordinates defined by the exponential map $\exp _{\xi}$ at $\xi$. We denote $\left|g_{\xi}(z)\right|:=\operatorname{det}\left({g_{\xi}}(z)_{i j}\right)$ and $\left(g_{\xi}^{i j}(z)\right)$ is the inverse matrix of $g_{\xi}(z)$. In particular, it holds

$$
g_{\xi}^{i j}(0)=\delta_{i j} \quad \text{ and } \quad \frac{\partial g_{\xi}^{i j}}{\partial z_{k}}(0)=0 \text { for any } i, j, k ,
$$
where $\delta_{i j}$ is the Kronecker symbol.
 Let $\xi_{0} \in M$ be fixed, and $\xi(z)\in B_g(\xi_0,r)$. Recall the expression of the scalar curvature $S_{g}(\xi_0)$ of $g$ at $\xi_0$, given by
$$
S_{g}\left(\xi_{0}\right)=S_{g}(\xi(0))=\sum_{i, j=1}^{N} \frac{\partial^{2} g_{\xi}^{i i}}{\partial z_{j}^{2}}(0)-\sum_{i, j=1}^{N} \frac{\partial^{2} g_{\xi}^{i j}}{\partial z_{i} \partial z_{j}}(0).
$$
    Assuming $M$ to be Ricci-flat, then we have $S_{g}(\xi)=0$ for all $\xi \in M$,  and 
\begin{equation}\label{def:tau}
\tau_g(\xi):=\sum_{i, j=1}^{n} \frac{\partial^{2} g_{\xi}^{i i}}{\partial z_{j}^{2}}(0)=\sum_{i, j=1}^{n} \frac{\partial^{2} g_{\xi}^{i j}}{\partial z_{i} \partial z_{j}}(0).
\end{equation}

\smallskip

Now we will discuss the asymptotical behavior of the approximate solution. 

\begin{remark}
    
Since the function $U$ is radial it follows that if $i\neq j$ then $\left\langle \Psi_{\ep}^i , \Psi_{\ep}^j \right\rangle_{\ep} =0$. 
Then, it is easy to see that for any $\xi \in M$ it holds
\[
\lim_{\ep \rightarrow 0} \left\langle Z_{\ep,\xi}^{i},\ Z_{\ep,\xi}^{j}\right\rangle_{\ep}= \delta_{ij} \int_{\mathbb{R}^{n}}\left(|\nabla\Psi^{i}|^{2}+|\Psi^{i}|^{2}\right) \ dz.
\]

Let us call $C=\displaystyle \int_{\mathbb{R}^{n}}(|\nabla\Psi^{i}|^{2}+(\Psi^{i})^{2}) \ dz$.
For $y\in B(0, r)\subset \RR^n$ we  set $\xi(y)=  \exp_{\xi}(y)\in B_g(\xi,r)$. Then, around $\xi$ it holds 
\begin{equation}\label{AA}
\lim_{\ep \rightarrow 0} \ep \left\Vert\frac{\partial W_{\ep,\xi(y)} }{\partial y_k}  \right\Vert_{\ep}  = C,
\quad
\lim_{\ep \rightarrow 0} \ep \left\langle Z^i_{\ep , \xi(y)}   , \frac{\partial W_{\ep,\xi(y)} }{\partial y_k}  \right\rangle_{\ep}  = \delta_{ik}C,
\end{equation}
and
\begin{equation}\label{DZ}
\lim_{\ep \rightarrow 0} \ep \left\Vert \frac{\partial Z^i_{\ep , \xi(y)} }{\partial y_k}  \right\Vert_{\ep}  = \displaystyle \int_{\mathbb{R}^{n}}\left(\left\vert\nabla 
\frac{\partial \Psi^{i}}{\partial z_k}\right\vert^{2}+\left(\frac{ \partial \Psi^{i}}{\partial z_k}\right)^{2}\right) \ dz.
\end{equation}
\end{remark}

To set the asymptotic expansion of the exact solution in Section \ref{sec:Expansion}, we give the well-known expansion  of the functions $g_{ij}$ in normal coordinates, which is proved, for instance, in \cite{G}.
\begin{lemma}\label{lema:Taylor}
	Let $(M,\g)$ be an analytic Riemannian manifold of dimension $n$. In a normal coordinates neighborhood of $\xi_{0} \in M$, the Taylor's series of $g$ around  $\xi_{0} $ is given by
\begin{align*}
        \sqrt{\left\lvert g_{\xi}(\varepsilon z)\right\rvert} &= 1 - \frac{\varepsilon^{2}}{4}\sum_{l,r,k=1}^{n} \frac{\partial^{2} g_{\xi}^{l l} }{\partial z_{r} \partial z_{k}} (0) z_{r} z_{k} + O(\varepsilon^{3} |z|^{3}),\\
        g_{\xi}^{i j}(\varepsilon z) &= \delta_{i j} + \frac{\varepsilon^{2}}{2} \sum_{r,k=1}^{n} \frac{\partial^{2} g_{\xi}^{i j} }{\partial z_{r} \partial z_{k}} (0) z_{r} z_{k} + O(\varepsilon^{3} |z|^{3}).
    \end{align*}	
\end{lemma}


\section{Finite dimensional reduction}\label{sec:Finite}

In this section, we use a fixed point theorem and the implicit function theorem to solve equation \eqref{Pii}.

\smallskip

Let $L_{\varepsilon, \xi}, N_{\varepsilon,\xi} : K_{\varepsilon,\xi}^{\perp} \rightarrow K_{\varepsilon, \xi}^{\perp}$    be the following operators
\begin{align*}
    L_{\varepsilon, \xi}(\phi) &:= \Pi_{\varepsilon, \xi}^{\perp} \left\{ \phi - \iota_{\varepsilon}^{*} \left[f'(W_{\varepsilon,\xi} ) \phi \right]  \right\}  \\
    N_{\varepsilon, \xi}(\phi) &:= \Pi_{\varepsilon, \xi}^{\perp} \left\{ \iota_{\varepsilon}^{*} \left[f(W_{\varepsilon,\xi} + \phi) - f(W_{\varepsilon,\xi} ) -f'(W_{\varepsilon,\xi} ) \phi \right]  \right\}.
\end{align*}
And $R_{\varepsilon,\xi} \in K_{\varepsilon,\xi}^{\perp}$ given by
\begin{equation}\label{def_R}
    R_{\varepsilon,\xi} := \Pi_{\varepsilon, \xi}^{\perp} \left\{ \iota_{\varepsilon}^{*} \left[f(W_{\varepsilon,\xi} ) \right] - W_{\varepsilon,\xi}  \right\},
\end{equation}   
so the equation \eqref{Pii} can be rewritten as
$$
L_{\varepsilon, \xi}(\phi) = N_{\varepsilon, \xi}(\phi) + R_{\varepsilon,\xi}.
$$
We will start this section by proving the invertibility of $L_{\varepsilon,\xi}$.
\begin{proposition}\label{invertible}
There exists $\varepsilon_{0} > 0 $ and $c>0$ such that for any $\xi \in M$ and for any $\varepsilon \in (0,\varepsilon_{0})$ 
$$ \| L_{\varepsilon,\xi} (\phi) \|_{\varepsilon} \geq c \| \phi \|_{\varepsilon}, \quad \text{ for any } \phi \in K_{\varepsilon,\xi}^{\perp}. $$
\end{proposition}
	\begin{proof} 
		We will proceed by contradiction. Since $M$ is compact, we may assume  that there are sequences $\{\ep_j\}_{j \in \mathbb{N}}$, with $\ep_j \rightarrow 0$ and $\{  \xi_j\}_{j \in \mathbb{N}}$, with $\xi_j\to \xi$ for some $\xi\in M$, with   
 $\{\phi_j \}\subset K^{\perp}_{\ep_j,  \xi_j}$, such that 
		$L_{\ep_j, {\xi}_j}(\phi_j)=\psi_j$, with $\|\phi_j\|_{\ep_j}=1$ and  $\|\psi_j\|_{\ep_j}\rightarrow 0.$
			
		Let $\zeta_j:=\Pi_{\ep_j,  \xi_j}\{ \phi_j -i^*_{\ep_j}[f'( W_{\ep_j,  \xi_j})\phi_j] \}$. Hence,
	\begin{equation}\label{phi_j}
	\phi_j -i^*_{\ep_j}[f'( W_{\ep_j,  \xi_j})\phi_j]=\psi_j+\zeta_j.
	\end{equation}
	That is, for each $j$, $\psi_j\in K^{\perp}_{\ep_j,  \xi_j}$ and $ \zeta_j \in K_{\ep_j,  \xi_j}$. Now, let $u_j:= \phi_j-(\psi_j+\zeta_j).$
		
		We will prove the following contradictory consequences of the existence of such series
		\begin{equation}
		\label{one}
		\frac{1}{\ep_j^n} \int_{M} f'(W_{\ep_j,  \xi_j}) u_j^2 \ dV_g\rightarrow 1,
		\end{equation}
		\noindent and 
		\begin{equation}
		\label{zero}
		\frac{1}{\ep_j^n} \int_{M} f'( W_{\ep_j,  \xi_j}) u_j^2 \  dV_g \rightarrow 0,	
		\end{equation}
		\noindent and then, such  sequences $\{  \xi_j\}$, $\{ \phi_j\}$, $\{\ep_j\}$ cannot exist.  We divide the proof in five steps. 
\bigskip 

 {\bf Step 1.} We start by proving 
\begin{equation}\label{zeta_j}
\Vert\zeta_{j}\Vert_{\ep_{j}}\rightarrow 0 \quad \hbox{ as }  j\rightarrow\infty. 
\end{equation}
Since $\zeta_{j}\in K_{\ep_{j}, {\xi_j}}$ , let $\zeta_{j} :=\displaystyle
\sum_{k=1}^{n}a_{j}^{k}Z_{\ep_{j},\xi_{j}}^{k}$. 
For any $l \in \{1,2,\dots, n\}$ we multiply $\psi_{j}+\zeta_j$ (equation (\ref{phi_j})) by $Z_{\ep_{j},\xi_{j}}^{l}$, 
and we find
\begin{equation}\label{product}
\begin{array}{lcc}
\displaystyle
\sum_{k=1}^{n}a_{j}^{k}\ \left\langle Z^k_{\ep_{j},\xi_{j} },\ Z_{\ep_{j},\xi_{j}}^{l}\right\rangle_{\ep_{j}}
&=& - 
\left\langle \io_{\ep_{j}}^{*}[ f'(W_{\ep_j,  \xi_j}) \phi_{j}],\ Z_{\ep_{j},\xi_{j}}^{l}\right\rangle_{\ep_{j}}. 
\end{array}
\end{equation}

On the other hand, we have that
\[
\displaystyle \sum_{k=1}^{n}a_{j}^{k}\ \left\langle 
Z_{\ep_{j},\xi_{j}}^{k},\ Z_{\ep_{j},\xi_{j}}^{l}\right\rangle_{\ep_{j}}=Ca_{j}^{l}+o(1).
\]
Indeed
\begin{align*}
\left\langle Z_{\ep_{j},\xi_{j}}^{k},\ Z_{\ep_{j},\xi_{j}}^{l}\right\rangle_{\ep_{j}}&=\frac{1}{\ep_j^{n}}\Big(\ep_j^{4}\int_{M}
\Delta_{g}Z_{\ep_{j},\xi_{j}}^{k}\Delta_{g}Z_{\ep_{j},\xi_{j}}^{l}  \ dV_g \\
& + b \
\ep_j^{2} \int_{M}\nabla Z_{\ep_{j},\xi_{j}}^{k}\nabla Z_{\ep_{j},\xi_{j}}^{l} \ dV_g+ a \ \int_{M}  \ Z_{\ep_{j},\xi_{j}}^{k}Z_{\ep_{j},\xi_{j}}^{l} \ dV_g\Big)\\
&=I_1+I_2+I_3\\
\end{align*}
where, because of the Taylor expansion of $g$ in normal coordinates given in \eqref{lema:Taylor}, we have
\begin{align*}
    I_1&=\ep_j^{4-n}\int_{B(0,r)} \left( \frac{1}{\sqrt{|g_{\xi_j}(z)|}} \sum_{\alpha,\beta=1}^{n} \partial_{\alpha} \left( g_{\xi_j}^{\alpha \beta}(z) \sqrt{|g_{\xi_j}(z)|} \partial_{\beta} \left( \Psi^{k}\left(\frac{z}{\ep_j}\right)\chi_{r}\left(\frac{z}{\ep_j}\right)\right) \right)\right)\\
    &\quad \times \left( \frac{1}{\sqrt{|g_{\xi_j}(z)|}} \sum_{\alpha,\beta=1}^{n} \partial_{\alpha} \left( g_{\xi_j}^{\alpha \beta}(z) \sqrt{|g_{\xi_j}(z)|} \partial_{\beta} \left( \Psi^{l}\left(\frac{z}{\ep_j}\right)\chi_{r}\left(\frac{z}{\ep_j}\right)\right) \right)\right) \sqrt{|g_{\xi_j}(z)|} \ dz \\
    &=\int_{\RR^{n}} \Delta \Psi^{k} \Delta \Psi^{l} \ dz + o(1).
\end{align*}
Analogously, we have
\[
I_2=b\int_{\RR^n} \nabla\Psi^k\nabla\Psi^l \ dz + o(1)
\quad
\text{ and }
\quad
I_3=a\int_{\RR^n} \Psi^k\Psi^l \ dz + o(1).
\]
Summing up these three terms, we obtain
\[
I_1+I_2+I_3=\int_{\mathbb{R}^n} \Delta\Psi^k\Delta\Psi^l+b\nabla\Psi^k\nabla\Psi^l + a \Psi^k\Psi^l \ dz + o(1)=\left\langle\Psi^k,\Psi^l \right\rangle + o(1)=\delta_{k,l}+o(1).
\]
because of the radial symmetry of the solution $U$.

On the other hand, combining this and (\ref{product}), we get
\begin{equation}\label{a_j1}
Ca_{j}^{l}+o(1)=\displaystyle\frac{1}{\ep_{j}^{n}}\displaystyle\int_{M}[\ep_{j}^{4}\Delta_g Z_{\ep_{j},\xi_{j}}^{l}\Delta_g\phi_{j}+b\ep_{j}^{2}\nabla Z_{\ep_{j},\xi_{j}}^{l}\nabla \phi_{j}+
 a Z_{\ep_{j},\xi_{j}}^{l}\phi_{j}-f'(W_{\ep_j,  \xi_j}) \phi_{j}Z_{\ep_{j},\xi_j}^{l}] \ dV_g. \\
\end{equation}

Let
\begin{equation*}
\tilde{\phi}_{j}(z)=
\left\{\begin{array}{lcc}
\phi_{j}(\exp_{\xi_{j}}(\ep_{j}z))\chi_{r}(\ep_{j}z) & \mathrm{i}\mathrm{f}\ z\in B(0,\ r/\ep_{j}),\\
0 & \mathrm{o}\mathrm{t}\mathrm{h}\mathrm{e}\mathrm{r}\mathrm{w}\mathrm{i}\mathrm{s}\mathrm{e}.
\end{array}\right.
\end{equation*}
Then we have that  for some constant $\tilde c$, $\Vert\tilde{\phi}_{j}\Vert_{H^{1}(\mathbb{R}^{n})}\leq \tilde c\Vert \tilde \phi_{j}\Vert_{\ep_{j}}\leq \tilde c$. 
Therefore, we can assume that $\tilde{\phi}_{j}$ converges weakly to some $\tilde{\phi}$  in $H^{2}(\mathbb{R}^{n})$ and 
 strongly in $L_{\operatorname{loc}}^{q}(\mathbb{R}^{n})$ for any $q\in[2,2^\sharp)$. 
Then, by equation (\ref{a_j1}), we have the following
\begin{align}
Ca_{j}^{l}+o(1)&=\displaystyle\frac{1}{\ep_{j}^{n}}\displaystyle\int_{M}[\ep_{j}^{4}\Delta_g Z_{\ep_{j},\xi_{j}}^{l}\Delta_g\phi_{j}+b\ep_{j}^{2}\nabla Z_{\ep_{j},\xi_{j}}^{l}\nabla \phi_{j}+
 a Z_{\ep_{j},\xi_{j}}^{l}\phi_{j}-f'(W_{\ep_j,  \xi_j}) \phi_{j}Z_{\ep_{j},\xi_j}^{l}] \ dV_g. \nonumber \\
&=\displaystyle\int_{\mathbb{R}^{n}}(\Delta\Psi^{l}\Delta\tilde{\phi}+b\nabla\Psi^{l}\nabla\tilde{\phi}+a\Psi^{l}\tilde{\phi}-f'(U)\Psi^{l}\tilde{\phi})\ dz  +o(1)=o(1).\label{a_j2}
\end{align}

From $(\ref{a_j2})$, we get that $a_{j}^{l}\rightarrow 0$ as $j\to\infty$, and then $(\ref{zeta_j})$ follows.

\bigskip

 {\bf Step 2.}
Now we are ready to prove (\ref{one}).
Recall that $u_j= \phi_j-(\psi_j+\zeta_j)$. 

Since $\Vert\phi_{j}\Vert_{\ep_{j}}=1, \Vert\psi_{j}\Vert_{\ep_{j}}\rightarrow 0$ and $\Vert\zeta_{j}\Vert_{\ep_{j}}\rightarrow 0$ by {\bf Step 1},
then
\begin{equation}\label{norma1}
\Vert u_{j}\Vert_{\ep_{j}}\rightarrow 1.
\end{equation}
Moreover, by (\ref{phi_j}) we know that $u_{j}=\io_{\ep_{j}}^{*}[f'(W_{\ep_{j},\xi_{j}})\phi_{j}]$, and hence, by (\ref{Adjoint}), it satisfies weakly
\begin{equation}\label{u_j solves}
\ep_{j}^{4}\Delta_{g}^2 u_{j}-b\ep_{j}^{2}\Delta_{g} u_{j}+a u_{j}=f'(W_{\ep_j,  \xi})u_{j}  
+f'(W_{\ep_j,  \xi})(\psi_{j}+\zeta_{j}) \quad \hbox{ in } \quad M.
\end{equation} 
Multiplying (\ref{u_j solves}) by $u_{j}$, and integrating over $M$, we get
\begin{equation}
\label{85}
\displaystyle \Vert u_{j}\Vert_{\ep_{j}}^{2}=\frac{1}{\ep_{j}^{n}}\int_{M}f'(W_{\ep_j,  \xi})u_{j}^{2} \ dV_g+
\frac{1}{\ep_{j}^{n}}\int_{M}f'(W_{\ep_j,  \xi})(\psi_{j}+\zeta_{j})u_{j} \ dV_g.
\end{equation}
By H\"{o}lder's inequality and equation (\ref{norma_q}), it follows that 
\begin{align}
     & |\displaystyle \frac{1}{\ep_{j}^{n}}\int_{M}f'(W_{\ep_j,  \xi_j})(\psi_{j}+\zeta_{j})u_{j} \ dV_g|
\\
&\leq \left(\displaystyle \frac{1}{\ep_{j}^{n}}\int_{M}(f'(W_{\ep_j,  \xi_j})\ u_{j})^2 \ dV_g\right)^{\frac{1}{2}} \left(\displaystyle \frac{1}{\ep_{j}^{n}}\int_{M}(\psi_{j}+\zeta_{j})^2\ dV_g\right)^{\frac{1}{2}}
\\
&\leq c  \  \|u_{j}\|_{\ep_j} \ \|\psi_{j}+\zeta_{j}\|_{\ep_j} = o(1).
\end{align}
Finally, since $\Vert\psi_{j}\Vert_{\ep_{j}}\rightarrow 0, \Vert\zeta_{j}\Vert_{\ep_{j}}\rightarrow 0,
$
and 
$\Vert u_{j}\Vert_{\ep_{j}}\rightarrow 1
$
as $j\rightarrow\infty$, then we get equation (\ref{one}). 
\bigskip

{\bf Step 3.} In this Step, for each $j$, we define the auxiliary functions
\begin{equation*}
\tilde{u}_{  j}=
    \left\{\begin{array}{lcc}
u_{j}\left(\exp_{\xi_{  j}}(\ep_{j}z)\right)  \ \chi_{r}\left(\exp_{\xi_{j}}(\ep_{j}z)\right) & \quad  z\in B(0,\ r/\ep_{j}),\\
0 & \mathrm{o}\mathrm{t}\mathrm{h}\mathrm{e}\mathrm{r}\mathrm{w}\mathrm{i}\mathrm{s}\mathrm{e},
\end{array}\right.
\end{equation*}
We claim that $\tilde{u}_j\to\tilde{u}$ weakly in $H^2(\RR^n)$ and strongly in $L^{p+1}(\mathbb{R}^n)$, where $\tilde{u}$ is a weak solution to the equation
\begin{equation}\label{u tilde}
\Delta^2\tilde{u}-b\Delta\tilde{u}+a\tilde{u}=f'(U)\tilde{u} \hspace{0.5cm} \hbox{ in } \mathbb{R}^{n}. 
\end{equation}

First note that each $u_{j}$ is compactly supported in $B_{g}(\xi_{j},\ r)$ and in consequence, $\tilde{u}_j$ is well defined. Moreover, it satisfies
$$
\Vert\tilde{u}_{  j}\Vert_{H^{2}(\mathbb{R}^{n})}^{2}\leq c\Vert {u}_{j}\Vert_{\ep_{j}}^{2}\leq c.
$$
Hence, up to a subsequence, $\tilde{u}_{  j}\rightarrow\tilde{u}$ weakly in $H^{2}(\mathbb{R}^{n})$ and 
strongly in $L_{\operatorname{loc}}^{q}(\mathbb{R}^{n})$ for any $q\in[1,2^\sharp)$, for some $\tilde u \in H^{2}(\mathbb{R}^{n})$. For any $\varphi \in C_0^{\infty}(\RR^n)$, we define
$$
\varphi_j(x):=\varphi\left(\frac{\exp^{-1}_{\xi_{j}}(x)}{\ep_j}\right)  \ \chi_{r}\left(\exp^{-1}_{\xi_{j}}(x)\right), \quad x\in B_g(\xi_{  j},\ep_j R) \subset M.
$$ 
Let us take $R$ such that $\mathit{supp} \  \varphi \subset B(0,R)$, and $j$ large enough such that $B_g(\xi_{  j},\ep_j R) \subset B_g(\xi_{j},r/2)$.  
Multiplying  $(\ref{u_j solves})$ by $\varphi_j$ and integrating over $M$, we have
\begin{align}\label{int}
 &\frac{1}{\ep_j^n} \int_M \left( \ep_j^4 \Delta_g u_j \ \Delta_g \varphi_j +b \ \ep_j^2 \nabla_g u_j \ \nabla_g \varphi_j + a \ u_j \varphi_j \right) d V_g\nonumber\\
&=\frac{1}{\ep_{j}^{n}}\int_{M}f'(W_{\ep_j,  \xi_j}) \ u_{j} \  \varphi_j \ dV_g+
\frac{1}{\ep_{j}^{n}}\int_{M} \ f'(W_{\ep_j,  \xi_j})(\psi_{j}+\zeta_{j}) \  \varphi_j  \ dV_g.
\end{align}

By setting $x=\exp_{\xi_{j}}(\ep_j z)$, we can integrate over $B(0,R)\subset\mathbb{R}^n$ in the following way
\begin{equation*}
\frac{1}{\ep_{j}^{n}}\int_{M}f'(W_{\ep_j,  \xi_j}) \ u_{j} \  \varphi_j \ dV_g= \int_{B(0,R)} f'\left( U(z) \chi_r(\ep_j z)  \right)\tilde u_{  j} \  \varphi \  |g_{\xi_{j}}(\ep_{j} z)|^{1/2} dz
\end{equation*}
and
\begin{equation*}
\frac{1}{\ep_{j}^{n}}\int_{M} \ f'(W_{\ep_j,  \xi_j})(\psi_{j}+\zeta_{j}) \  \varphi_j  \ dV_g= \int_{B(0,R)} f'\left( U(z) \chi_r(\ep_j z) \right)(\tilde \psi_j +\tilde \zeta_j) \  \varphi \  |g_{\xi_{j}} (\ep_{j} z)|^{1/2} dz, 
\end{equation*}

\noindent where $\tilde \psi_{j}(z):=\psi_{j} (\exp_{\xi_{j}}(\ep_j \ z))$ and $\tilde \zeta_{j}(z):=\zeta_{j} (\exp_{\xi_{j}}(\ep_j \ z))$ for $z \in B(0, R/\ep_j)$.

Once again, H\"older's inequality comes to the rescue and we get
\begin{align*}
    \int_{B(0,R)} \tilde u_{  j} \varphi |g_{\xi_{j}}(\ep_{j} z)|^{1/2} dz &\leq  \ep_j^2    \left(\int_{B(0,R)} \tilde u_{  j}^2  dz\right)^{1/2}  \left(\int_{B(0,R)} \varphi^2 |g_{\xi_{j}}(\ep_{j} z)| dz\right)^{1/2} \\
    &\leq c \  \ep_j^2    \|\tilde u_{  j}\|_{H^2(\RR^n)} =o(\ep_{j}).
\end{align*}

\noindent with $c$ an upper bound for $\int_{B(0,R)} \varphi^2 |g_{\xi_{j}}(\ep_{j} z)| dz$. Recall also that $u_{  j}$ is bounded independently of $j$ in $H^2(\RR^n)$.
Hence, taking the limit in (\ref{int}) we have

\begin{equation}
\label{weaklysolves}
\lim_{j\to\infty}\frac{1}{\ep_j^n} \int_M \left( \ep_j^4 \Delta_g u_j \ \Delta_g \varphi_j +b \ \ep_j^2 \nabla_g u_j \ \nabla_g \varphi_j + a \ u_j \varphi_j \right) dV_g = \int_{\RR^n} f'\left( U(z)  \right) \tilde u \  \varphi \  dz,
\end{equation} 
 \noindent since $\tilde \psi_j, \tilde \zeta_j \rightarrow 0$ weakly in $H^2(\RR^n)$ and strongly in $L^{p+1}(\mathbb{R}^n)$. Equation (\ref{weaklysolves}) proves the claim. 

\bigskip

{\bf Step 4.}
We now claim that for any $k \in \{1,2,...n\}$, $\tilde u$ satisfies also
\begin{equation}
\label{productu}
\int_{\mathbb{R}^{n}}\left(\Delta\Psi^{k} \Delta \tilde{u} +\nabla\Psi^{k}\nabla\tilde{u}+\Psi^{k}\tilde{u}\right) \ dz=0.
\end{equation}

Since $\phi_{j}, \psi_{j}\in K_{\ep_{j},  \xi_{j}}^{\perp}$, from equation (\ref{zeta_j}) we have
\begin{equation}\label{langle}
|\left\langle Z_{\ep_{j},\xi_{  j}}^{k},\ u_{j}\right\rangle _{\ep_{j}}|
= | \left\langle Z_{\ep_{j},\xi_{  j}}^{k},\phi_{j} -\psi_j-\zeta_j \right\rangle_{\ep_{j}} | 
= |\left\langle Z_{\ep_{j},\xi_{  j}}^{k},\ \zeta_{j}\right\rangle_{\ep_{j}} |
\leq \Vert Z_{\ep_{j},\xi_{  j}}^{k}\Vert_{\ep_{j}} \Vert\zeta_{j}\Vert_{\ep_{j}}=o(1).
\end{equation}

 On the other hand, we have
\begin{align}
\left\langle Z_{\ep_{j},\xi_{  j}}^{k},\ u_{j}\right\rangle_{\ep_{j}}&=\frac{1}{\ep_{j}^n}\int_{M}
\left(\ep_{j}^{4} \ \Delta_{g}Z_{\ep_j ,\xi_{  j} }^{k}\Delta_g u_j + b \ \ep_{j}^{2} \ \nabla Z_{\ep_j ,\xi_{  j} }^{k}\nabla_g u_j+ \
a \ Z_{\ep_{j},\xi_{  j}}^{k} \ u_{j}\right) \ dV_g\nonumber\\
&\leq c\displaystyle \int_{\mathbb{R}^{n}}\left(\Delta\Psi^{k}\Delta\tilde{u}+\nabla\Psi^{k}\nabla\tilde{u}+\Psi^{k}\tilde{u}\right) \ dz+o(1). \label{tilde} 
\end{align}
From $(\ref{langle})$ and $(\ref{tilde})$ we prove the claim of equation (\ref{productu}).
Therefore, $(\ref{u tilde})$ and $(\ref{productu})$ imply that $\tilde{u}=0$, because we assume $U$ to be non-degenerate.  

\bigskip
{\bf Step 5.} In this final Step, we deduce equation (\ref{zero}). 
Since $\tilde u=0$ by {\bf Step 4}, we know that
\begin{equation*}
\frac{1}{\ep_{j}^{n}}\int_{M}f'(W_{\ep_{j},  \xi_{j}}) \ u_{j}^{2} \ dV_g
=\int_{{B(0, \ \ep_j  r)}}  f'\left(U(z) \chi_{r}(\ep_{j} z)  \right)\tilde u_{  j}^2(z) \  |g_{\xi_{j}}(\ep_{j} z)|^{1/2} \ dz
= o(1).
\end{equation*}
This last equation proves (\ref{zero}), which contradicts (\ref{one}) and concludes the proof.
\end{proof} 
\smallskip

Next, we study the term estimate of $R_{\ep, \xi}=\Pi^{\perp}_{\ep, \xi}\{ i^*_{\ep}[f(W_{\ep, \xi})]-W_{\ep, \xi} \}.$


\begin{remark}\label{Lapl. in normal coord} Let $v(z) := u (\exp_\xi (z))$, $z \in B(0, r )$, $u\in H_g^2(M)$. From \cite[Remark 3.2]{micheletti2009role} we know that
$$
 \Delta_{g_{\xi}} v=-\Delta v+A^{ij} \partial_{i j}^{2} v+B^{k} \partial_{k} v 
$$
where
$$
 A^{ij}(z):=-\left[g^{i j}(z)-\delta^{ij}(z)\right],
\text{ and } \quad B^{k}(z):=g^{i j}(z) \Gamma_{i j}^{k}(z).
$$
Here $\Delta$ is the Euclidean Laplacian and we are using the Einstein summation convention.
Then, we have
$$
\begin{aligned}
\Delta_{g_\xi}^2 v=\Delta_{g_\xi}[- & \left.\Delta v+A^{st} \partial_{s t}^{2} v+B^{h} \partial_{h} v\right] \\
= & -\Delta\left[-\Delta v+A^{st} \partial_{s t}^{2} v+B^{h} \partial_{h} v\right]+A^{ij} \partial_{i j}^{2}\left[-\Delta v+A^{st} \partial_{s t}^{2} v+B^{h} \partial_{h} v\right] \\
& +B^{k} \partial_{k}\left[-\Delta v+A^{st} \partial_{s t}^{2} v+B^{h} \partial_{h} v\right].
\end{aligned}
$$
\end{remark}
We also need to recall the formula for the bi-laplacian of a product function, which is
\[
   \Delta^2 (U_\ep\chi_r) = \Delta^2 (U_\ep)\chi_r + 2\Delta(U_\ep) \Delta\chi_{r}+  4\partial^3_{ijj} U_\ep\partial_i \chi_r
+4\partial_{ij}^2 U_\ep \partial_{ij}^2\chi_r + 4 U_\ep\partial^3_{ijj} \chi_r + U_\ep\Delta^2 (\chi_r).
\] 
Now, we are ready to prove the following result.

\begin{lemma}
	\label{estimateR}
 There exist $\ep_0>0$ and $c>0$ such that for any $\ep\in(0,\ep_0)$, it holds 
 \[
 \Vert R_{\ep, \xi}\Vert_\ep \leq c \ep^2.
 \]
\end{lemma}
\begin{proof}
As $\io_{\ep}^{*}:L_{g,\ep}^{\frac{p+1}{p}}\rightarrow H^{2}_{\ep}$ is a surjective map, we define the function $V_{\ep,\xi}$ on $M$ by $W_{\ep ,\xi} = \iota_{\varepsilon}^* (V_{\ep,\xi})$  and consider $\tilde{V}_{\ep, \xi} (z) := V_{\ep,\xi}\left(\exp_{\xi}(z)\right) $ on $B(0,r) \subseteq \mathbb{R}^{n}$ . In other words, it satisfies
\begin{align*}
    \tilde{V}_{\ep, \xi} (z) &= \ep^{4}\Delta_{g}^{2} W_{\ep ,\xi}\left(\exp_{\xi}(z)\right) -\ep^2 b \Delta_{g} W_{\ep ,\xi}\left(\exp_{\xi}(z)\right) + a W_{\ep,\xi}\left(\exp_{\xi}(z)\right).
\end{align*}
Denote also  $\tilde{W}_{\ep,\xi}(z)=W_{\varepsilon,\xi}(\exp_{\xi}z) = U_{\varepsilon}(z)\chi_r(z) $, where $v_\ep(z):=v(\frac{\ep}{z})$ (and in consequence $\partial_iU_\ep(z)=\frac{1}{\ep}\partial_i U(\frac{z}{\ep})$).
By Remark \ref{Lapl. in normal coord}, we can compute
\begin{align*}
    \varepsilon^{2}\Delta_{g_\xi} \tilde{W}_{\ep ,\xi}(z) &= \ep^2\Delta \left(U_{\ep} \chi_{r}\right)(z) - \ep^2 A^{ij} \partial^2_{ij} \left(U_{\ep} \chi_{r}\right)(z) + \ep^2 B^{k}\partial_{k} \left(U_{\ep} \chi_{r}\right)(z) \\
     &= \Delta U \left(\frac{z}{\ep}\right)  \chi_{r}(z)  + 2 \ep^2 \partial_{i} U_{\ep}(z)  \partial_{i} \chi_{r}(z)  + \ep^2 U_{\ep} (z) \Delta \chi_{r}(z)  \\
     &\quad + \ep^2 B^{k}\partial_{k} \left(U_{\ep} \chi_{r}\right)(z)
     - \ep^2 A^{ij} \partial^2_{ij} \left(U_{\ep} \chi_{r}\right)(z).
\end{align*}
In the same spirit, we compute
\begin{align*}
    \varepsilon^{4}\Delta_{g_\xi}^2 \tilde{W}_{\ep ,\xi}(z) &= \ep^4\Delta^2 \left(U_{\ep} \chi_{r}\right) (z)- \ep^4\Delta \left(A^{st} \partial_{s t}^{2}  (U_{\ep} \chi_{r}) \right)(z) - \ep^4\Delta \left(B^{h} \partial_{h} (U_{\ep} \chi_{r})\right)(z) \\
    &\quad - \ep^4 A^{ij} \partial_{i j}^{2} \left(\Delta(U_{\ep} \chi_{r})\right)(z) +\ep^4 A^{ij} \partial_{i j}^{2} \left( A^{st} \partial_{s t}^{2}  (U_{\ep} \chi_{r}) )\right) + \ep^4 A^{ij} \partial_{i j}^{2} \left(B^{h} \partial_{h} (U_{\ep} \chi_{r})\right) \\
    &\quad - \ep^4 B^{k} \partial_{k}\left(\Delta (U_{\ep} \chi_{r}) \right)+ \ep^4 B^{k} \partial_{k} \left(A^{st} \partial_{s t}^{2} (U_{\ep} \chi_{r})\right) + \ep^4 B^{k} \partial_{k}\left(B^{h}\partial_{h}\left(U_{\ep} \chi_{r}\right)\right)\\
      &= \Delta^2 U\left(\frac{z}{\ep}\right)\chi_r(z) + 2\ep^4\Delta U_\ep (z) \Delta\chi_{r}(z)+  4\ep^4 \partial^3_{ijj} U_\ep(z)\partial_i \chi_r(z)\\
&\quad+4\ep^4\partial_{ij}^2 U_\ep \partial_{ij}^2\chi_r (z)+ 4\ep^4 U_\ep(z)\partial^3_{ijj} \chi_r(z) +\ep^4 U_\ep(z)\Delta^2 \chi_r(z)\\
      &\quad - \ep^4\Delta \left(A^{st} \partial_{s t}^{2}  (U_{\ep} \chi_{r}) \right)(z) - \ep^4\Delta \left(B^{h} \partial_{h} (U_{\ep} \chi_{r})\right)(z) \\
    &\quad - \ep^4 A^{ij} \partial_{i j}^{2} \left(\Delta(U_{\ep} \chi_{r})\right)(z) +\ep^4 A^{ij} \partial_{i j}^{2} \left( A^{st} \partial_{s t}^{2}  (U_{\ep} \chi_{r}) )\right) + \ep^4 A^{ij} \partial_{i j}^{2} \left(B^{h} \partial_{h} (U_{\ep} \chi_{r})\right) \\
    &\quad - \ep^4 B^{k} \partial_{k}\left(\Delta (U_{\ep} \chi_{r}) \right)+ \ep^4 B^{k} \partial_{k} \left(A^{st} \partial_{s t}^{2} (U_{\ep} \chi_{r})\right) + \ep^4 B^{k} \partial_{k}\left(B^{h}\partial_{h}\left(U_{\ep} \chi_{r}\right)\right).
\end{align*}
Taking into account the fact that the function $U$ satisfies the limit equation and the previous computations, it follows that
\[
    \tilde{V}_{\ep, \xi} (z)=U_\ep^{p}(z) \chi_{r}(z) + \mathcal{R}_\ep(U_{\ep},\chi_{r}),
    \]
    where
\begin{align*}
    \mathcal{R}_\ep(U_{\ep},\chi_{r})&=2\ep^4\Delta U_\ep(z) \Delta\chi_{r}(z)+  4\ep^4 \partial^3_{ijj} U_\ep(z)\partial_i \chi_r(z)\\
&\quad+4\ep^4\partial_{ij}^2 U_\ep(z) \partial_{ij}^2\chi_r (z)+ 4\ep^4 \partial_{i}U_\ep(z)\partial^3_{ijj} \chi_r(z) +\ep^4 U_\ep(z)\Delta^2 \chi_r(z)\\
      &\quad - \ep^4\Delta \left(A^{st} \partial_{s t}^{2}  (U_{\ep} \chi_{r}) \right)(z) - \ep^4\Delta \left(B^{h} \partial_{h} (U_{\ep} \chi_{r})\right)(z) \\
    &\quad - \ep^4 A^{ij} \partial_{i j}^{2} \left(\Delta(U_{\ep} \chi_{r})\right)(z) +\ep^4 A^{ij} \partial_{i j}^{2} \left( A^{st} \partial_{s t}^{2}  (U_{\ep} \chi_{r}) )\right)(z) + \ep^4 A^{ij} \partial_{i j}^{2} \left(B^{h} \partial_{h} (U_{\ep} \chi_{r})\right)(z) \\
    &\quad - \ep^4 B^{k} \partial_{k}\left(\Delta (U_{\ep} \chi_{r}) \right)(z)+ \ep^4 B^{k} \partial_{k} \left(A^{st} \partial_{s t}^{2} (U_{\ep} \chi_{r})\right)(z) + \ep^4 B^{k} \partial_{k}\left(B^{h}\partial_{h}\left(U_{\ep} \chi_{r}\right)\right)(z)\\
    &\quad -2b \ep^2\partial_{i} U_\ep(z) \partial_{i} \chi_{r}(z)  
    - b \ep^2 U_\ep(z) \Delta \chi_{r}(z)  
    -b\ep^2 B^{k}\partial_{k} \left(U_{\ep} \chi_{r}\right)(z)
    +b \ep^2 A^{ij} \partial^2_{ij} \left(U_{\ep} \chi_{r}\right)(z). 
\end{align*}

On the other hand, by \eqref{def_R} and \eqref{norma_ep}, we get there exists a positive constant $C$ such that for $\ep>0$ small and for any point $\xi \in M$, there holds,
\begin{equation}\label{R-estimate1}
\|R_{\ep,\xi}\|_\ep = \|\iota_\ep^*(f(W_{\ep,\xi})-V_{\ep,\xi})-W_{\ep,\xi}+\iota_\ep^*(V_{\ep,\xi})\|_\ep \leq  C |f(W_{\ep, \xi})-V_{\ep,\xi}|_{\frac{p+1}{p},\ep}.
\end{equation}
We need to point out that for some positive constant $c$ we have
\begin{align}\label{R-integral}
  \int_{B_g(\xi,r)}  |f(W_{\ep, \xi})-V_{\ep,\xi}|^\frac{p+1}{p} \ dV_g&\leq
  c\int_{B(0,r)}  |f(\tilde{W}_{\ep, \xi})-\tilde{V}_{\ep,\xi}|^\frac{p+1}{p} dz\\ \nonumber
  &\leq c\int_{B(0,r)} \left|U_\ep^{p}(z) (\chi_{r}^{p}(z) - \chi_{r}(z))\right|^{\frac{p+1}{p}} \ dz\\ \nonumber
  &+ c\int_{B(0,r)} \left|\mathcal{R}_\ep(U_{\ep},\chi_{r})\right|^{\frac{p+1}{p}} \ dz. \nonumber
\end{align}

We are led to estimate each term on the right-hand side in \eqref{R-integral}. 
However, by arguing as in \cite{micheletti2009role},  it is easy to see that
\begin{equation}\label{R-estimate2}
\int_{B(0,r)} \left|\mathcal{R}_\ep(U_{\ep},\chi_{r}) \right|^{\frac{p+1}{p}}\ dz=o(\ep^{n+2\frac{p+1}{p}}).
\end{equation}
We write some of the estimates in the Section \ref{sec:Appendix}.
Hence, from \eqref{R-estimate1}, \eqref{R-integral} and \eqref{R-estimate2}, the following estimates hold
\begin{align*}
    \|R_{\ep,\xi}\|_\ep &\leq  \frac{C}{\ep^{n}} \int_{B(0,r)} \left|U_\ep^{p}(z) (\chi_{r}^{p}(z) - \chi_{r}(z))\right|^{\frac{p+1}{p}} dz + o(\ep^{2\frac{p+1}{p}}).
\end{align*}
Using again  the exponential decay of $U$ and the definition of $\chi_r$, we have
\[
\int_{B(0, r)}\left|U_{\varepsilon}^{p}(z)\left(\chi_{r}^{p}(z)-\chi_{r}(z)\right)\right|^{\frac{p+1}{p}} d z \leq \int_{B(0, r) \backslash B(0, r / 2)} U_{\varepsilon}^{p}(z) d z\leq o\left(\varepsilon^{n+2 (p+1)/p}\right).
\]
Since $1<p<p^\sharp$ then $\ep^{2\frac{p+1}{p}}<\ep^2$ and we get the estimation claimed. 
\end{proof}

Next, we prove the main result of this section.

\begin{proposition}\label{prop1}  
    There exists $\varepsilon_{0}>0$ and $c>0$ such that for any $\xi \in  M$ and for all $\varepsilon \in\left(0, \varepsilon_{0}\right)$ there exists a unique $\phi_{\varepsilon, \xi}=\phi(\varepsilon, \xi) \in K_{\varepsilon, \xi}^{\perp}$ which solves \eqref{Pii}. Moreover,
$$
\left\|\phi_{\varepsilon, \xi}\right\|_{\varepsilon}<c \varepsilon^2,
$$
and $\xi \mapsto \phi_{\varepsilon, \xi}$ is a $C^{1}$ map.
\end{proposition}

\begin{proof} We argue exactly as in \cite[Proposition 3.5]{micheletti2009role}. For the reader's convenience, we briefly recall the main steps. 
We solve \eqref{Pii} by a fixed point argument. 
    Recall that equation \eqref{Pii} is equivalent to
\begin{equation}\label{L_ep}
    L_{\varepsilon, \xi}(\phi) = N_{\varepsilon, \xi}(\phi) + R_{\varepsilon,\xi}. 
\end{equation}
We define the operator $T_{\varepsilon, \xi} : K_{\varepsilon, \xi}^{\perp} \rightarrow K_{\varepsilon, \xi}^{\perp}$ by
$$
T_{\varepsilon, \xi}(\phi) =L_{\varepsilon, \xi}^{-1}\left(N_{\varepsilon, \xi}(\phi)+R_{\varepsilon, \xi}\right).
$$
By Proposition \ref{invertible}, $T_{\varepsilon, \xi}$ is well defined and it holds
\begin{equation*}
\left\|T_{\varepsilon, \xi}(\phi)\right\|_{\varepsilon} \leq c\left(\left\|N_{\varepsilon, \xi}(\phi)\right\|_{\varepsilon}+\left\|R_{\varepsilon, \xi}\right\|_{\varepsilon}\right) 
\end{equation*}
for some constant $c>0$. Moreover, by the linearity of $L_{\varepsilon, \xi}^{-1}$, the properties of $\iota^*$, and the Mean Value Theorem we get
\begin{align*}
\left\|T_{\varepsilon, \xi}\left(\phi_{1}\right)-T_{\varepsilon, \xi}\left(\phi_{2}\right)\right\|_{\varepsilon} & \leq c\left(\left\|N_{\varepsilon, \xi}\left(\phi_{1}\right)-N_{\varepsilon, \xi}\left(\phi_{2}\right)\right\|_{\varepsilon}\right)\\
&=c \ \Pi_{\varepsilon, \xi}^{\perp} \left\{ \iota_{\varepsilon}^{*} \left[f(W_{\varepsilon,\xi} + \phi_1) - f(W_{\varepsilon,\xi} + \phi_1)  -f'(W_{\varepsilon,\xi} ) \phi_1 + f'(W_{\varepsilon,\xi} ) \phi_2 \right]  \right\} \\
&\leq c\left|f(W_{\varepsilon,\xi} + \phi_1) - f(W_{\varepsilon,\xi} + \phi_2)  -f'(W_{\varepsilon,\xi} ) \phi_1 + f'(W_{\varepsilon,\xi} ) \phi_2\right|_{\frac{p+1}{p},\ep}\\
& \leq c\left|f^{\prime}\left(W_{\varepsilon, \xi}+\phi_{2}+t\left(\phi_{1}-\phi_{2}\right)\right)-f^{\prime}\left(W_{\varepsilon, \xi}\right)\right|_{\frac{p+1}{p},\ep}\left\|\phi_{1}-\phi_{2}\right\|_{\frac{p+1}{p},\ep}\\
& \leq c\left|f^{\prime}\left(W_{\varepsilon, \xi}+\phi_{2}+t\left(\phi_{1}-\phi_{2}\right)\right)-f^{\prime}\left(W_{\varepsilon, \xi}\right)\right|_{\frac{p+1}{p},\ep}\left\|\phi_{1}-\phi_{2}\right\|_{\varepsilon}.
\end{align*}
Here, $c$ denotes any positive constant. 

\begin{equation*}
\begin{array}{lcc}
    \mid 
 f(W_{\ep,\xi}+\phi_1)-f(W_{\ep,\xi}+\phi_2)-f'(W_{\ep,\xi})(\phi_1-\phi_2) \mid_{\frac{p+1}{p},\ep} \\
 \leq C \mid 
 (f'(W_{\ep,\xi}+\phi_2 + \tau(\phi_1-\phi_2))-f'(W_{\ep,\xi}))(\phi_1-\phi_2) \mid_{\frac{p+1}{p},\ep} \\
  \leq C \mid 
 f'(W_{\ep,\xi}+\phi_2 + \tau(\phi_1-\phi_2))-f'(W_{\ep,\xi})\mid_{\frac{p+1}{p-1},\ep}\mid \phi_1-\phi_2 \mid_{\frac{p+1}{p},\ep}. \\
\end{array}
\end{equation*}

Recall (from \cite[Lemma 2.2]{yanyan1998}) that for all $a>0, \ b \in \RR$, we have
\begin{equation*}
	\left|\lvert a +  b\rvert^q- a^q\right| \le 
	\left\{
	\begin{array}{rcl}
	& 	C(q)\min\left\{\lvert b\rvert^q, \ a^{q-1}\lvert b\rvert\right\} \quad &\text{if } q\in\left(0,1\right)\\
	&	C(q)\left( a^{q-1}\lvert b\rvert + \lvert b \rvert^q \right)\quad &\text{if } q\geq 1,
	\end{array}
	\right.
\end{equation*}
	which implies
\begin{equation}\label{ineq:f'(W)}
	\lvert f'(W_{\ep,\xi}+\phi)-f'(W_{\ep,\xi})\rvert \le 
	\left\{
	\begin{array}{rcl}
	& 	C(q)\lvert \phi\rvert^{p-1} \quad &\text{if } p\in\left(1,2\right)\\
	& C(q) \left(	W_{\ep,\xi}^{p-2}\lvert\phi\rvert + \lvert\phi\rvert^{p-2} \right) \quad &\text{if } p\geq 2.
	\end{array}
	\right.
		\end{equation}
  Moreover, it is easy to check that
\[
\lvert \lvert \phi\rvert^{p-1}\rvert_{\frac{p+1}{p}}\leq c \lvert\phi\rvert^{p-1}_{p,\ep},
\]
  and, if $p\geq 2$
  \[
  \lvert W_{\ep, \xi}^{p-2}\phi^2\rvert_{\frac{p+1}{p},\ep}\lvert\phi\rvert_{p+1,\ep}^2\leq C\|\phi\|_\ep^2.
\]
From these facts and from \eqref{ineq:f'(W)}, it follows that 

\begin{equation*}
\lvert f'(W_{\ep,\xi}+\phi_2 + \tau(\phi_1-\phi_2))-f'(W_{\ep,\xi})\rvert_{\frac{p+1}{p-1},\ep} 
 \leq C \Vert \phi_1 -\phi_2\Vert_\ep.
\end{equation*}

  Moreover, with the same estimates, we get
$$
\left\|N_{\varepsilon, \xi}(\phi)\right\|_{\varepsilon} \leq c\left(\|\phi\|_{\varepsilon}^{2}+\|\phi\|_{\varepsilon}^{p-1}\right).
$$
So, we deduce

\begin{equation}
\label{fixed}
\begin{array}{lcc}
\Vert T_{\ep,\xi}(\phi_1) -  T_{\ep,\xi}(\phi_2)\Vert_\ep 
\leq \Vert N_{\ep,\xi}(\phi_1)- N_{\ep,\xi}(\phi_2)\Vert_\ep \leq c \Vert \phi_1 -\phi_2\Vert_\ep,\\
\end{array}
\end{equation}
for $c \in (0,1)$, provided $\|\phi_1\|_{\ep }$ and $\|\phi_2\|_{\ep }$ are small enough.
This fact, combined with Lemma \ref{estimateR}, gives us

\begin{equation*}
    \left\|T_{\varepsilon, \xi}(\phi)\right\|_{\varepsilon} \leq c\left(\left\|N_{\varepsilon, \xi}(\phi)\right\|_{\varepsilon}+\left\|R_{\varepsilon, \xi}\right\|_{\varepsilon}\right) \leq c\left(\|\phi\|_{\varepsilon}^{2}+\|\phi\|_{\varepsilon}^{p-1}+\varepsilon^2\right) .
\end{equation*}

So, there exists $c\in(0,1)$ such that $T_{\varepsilon, \xi}$ maps a ball of center 0 and radius $c \varepsilon^2$ in $K_{\varepsilon, \xi}^{\perp}$ into itself and it is a contraction. So, there exists a fixed point $\phi_{\varepsilon, \xi}$ with the norm $\left\|\phi_{\varepsilon, \xi}\right\|_{\varepsilon} \leq \varepsilon^2$.

Finally, to prove that the map $\xi \rightarrow \phi_{\ep ,\xi}$ is in fact a $C^1$ map, we use the Implicit Function Theorem applied to the function
$$
G : M \times H_{\varepsilon}^2\longrightarrow \mathbb{R},\;(\xi, u) \longmapsto
G(\xi, u):= T_{\ep,\xi}(u)- u.
$$

As mentioned previously, equation (\ref{fixed}) ensures the existence of a $\phi_{\ep ,\xi}$ such that $G(\xi, \phi_{\ep ,\xi})=0$. Moreover, when restricted to a sufficiently small ball, $T_{\ep,\xi}(\phi)$ is differentiable and has a differentiable inverse $L_{\ep,\xi}(\phi)$, by  \eqref{L_ep}. Consequently, according to the Implicit Function Theorem, the mapping $\xi \rightarrow \phi_{\ep,\xi}$ is $C^1$.

We have that $G\left(\xi, \phi_{\varepsilon, \xi}\right)=0$ and that the operator $\frac{\partial G}{\partial u} \left(\xi, \phi_{\varepsilon, \xi}\right): H_{\varepsilon}^2\rightarrow H_{\varepsilon}^2$ is invertible. This concludes the proof.
\end{proof} 


\section{Asymptotic Expansion}\label{sec:Expansion}

In this section, we will prove some important properties regarding the functional $
J_{\varepsilon}:H_\ep^2\to\RR$, defined by
$$
J_{\varepsilon} (u) := \frac{1}{\varepsilon^{n}}  \int_{M} \left(\dfrac{\varepsilon^4}{2} |\Delta_{g} u(x) |^{2} + \dfrac{\varepsilon^2}{2} b|\nabla  u(x)|^{2} + \dfrac{a}{2} |u(x)|^{2}  - \frac{1}{(p+1)} (u_+(x))^{p+1} \right) dV_{g}(x),      
$$
where $u_+ (x) = \max\{u(x), 0\}$.
It is well known that any critical point of the functional $J_{\varepsilon}$ corresponds to a solution of the problem (\ref{equivalent problem}).

\begin{proposition}\label{prop2}
    For $\xi \in M$, and $\ep>0$ small, the following expansion holds.
    $$
        J_{\varepsilon} (W_{\varepsilon,\xi}) = \alpha\ + \ \beta \varepsilon^2 \ \tau_g(\xi) \ + \ o(\varepsilon^{2}),
    $$
    where
    \begin{equation}\label{alfa}
	\alpha:= \frac{1}{2} \displaystyle\int_{\mathbb{R}^n} \left(|\Delta U(z)|^2  + |\nabla U(z)|^2  + U^2(z) \right) \ dz - \frac{1}{p+1} \displaystyle\int_{\mathbb{R}^n} U^{p+1}(z) \ dz,
	\end{equation}
 \begin{equation}\label{beta}
	\beta:= \displaystyle\int_{\mathbb{R}^n} \left(\frac{U'(|z|)}{|z|} - U''(|z|) \right)^2  z_1^2 \ z_2^2\ dz + \frac{b}{2}\displaystyle\int_{\mathbb{R}^n} U'(|z|)^{2}z_1^2\ z_2^2 \ dz,
	\end{equation}
 and $\tau_g$ as in \eqref{tau}.
\end{proposition}

\begin{proof}
    We will start by estimating 
    $    \varepsilon^{4-n} \int_{M} |\Delta_{g} W_{\varepsilon,\xi}(x) |^{2} \ dV_{g}(x).
    $
Straightforward calculations show 
\begin{align*}
  &\ep^{4-n}\int_{M} |\Delta_{g} W_{\varepsilon,\xi}(x) |^{2}  dV_{g}(x)   \\
    &=  \ep^{4-n}\int_{B(0,r)}  \frac{1}{\sqrt{|g_{\xi}(z)|}} \left[\sum_{i,j=1}^{n} \partial_{i} \left( g_{\xi}^{i j}(z) \sqrt{|g_{\xi}(z)|} \left( \frac{1}{\varepsilon} \partial_{j}  U\left(\frac{z}{\varepsilon}\right) \chi_r(z) + U\left(\frac{z}{\varepsilon}\right)  \partial_{j} \chi_r (z)\right) \right)\right]^{2}  \ dz\\
    &\quad =  \ep^{4-n}\int_{B(0,r)}  \left[ \frac{1}{\sqrt{|g_{\xi}(z)|}} \sum_{i,j=1}^{n} \partial_{i} \left( g_{\xi}^{i j}(z) \sqrt{|g_{\xi}(z)|} \right) \left( \frac{1}{\varepsilon} \partial_{j}  U\left(\frac{z}{\varepsilon}\right) \chi_r(z) + U\left(\frac{z}{\varepsilon}\right)  \partial_{j} \chi_r (z)\right) \right.\\
    &\quad  +\left. \frac{1}{\sqrt{|g_{\xi}(z)|}} \sum_{i,j=1}^{n}  g_{\xi}^{i j}(z) \sqrt{|g_{\xi}(z)|} \,\,\partial_{i} \left( \frac{1}{\varepsilon} \partial_{j}  U\left(\frac{z}{\varepsilon}\right) \chi_r(z) + U\left(\frac{z}{\varepsilon}\right)  \partial_{j} \chi_r (z)\right) \right]^{2}  \sqrt{|g_{\xi}(z)|}\ dz
    \end{align*}
    \begin{align*}
     &\quad  =  \ep^{4-n}\int_{B(0,r)}  \left[ \frac{1}{\sqrt{|g_{\xi}(z)|}} \sum_{i,j=1}^{n} \partial_{i} \left( g_{\xi}^{i j}(z) \sqrt{|g_{\xi}(z)|} \right)
     \left( \frac{1}{\varepsilon} \partial_{j}  U\left(\frac{z}{\varepsilon}\right) \chi_r(z) + U\left(\frac{z}{\varepsilon}\right)  \partial_{j} \chi_r (z)\right) \right.\\    &\quad  +\left.  \sum_{i,j=1}^{n}  g_{\xi}^{i j}(z) \,\, \left( \frac{1}{\varepsilon^2} \partial^2_{ij}  U\left(\frac{z}{\varepsilon}\right) \chi_r(z)+\frac{2}{\varepsilon} \partial_{j}  U\left(\frac{z}{\varepsilon}\right)\partial_i \chi_r(z)  + U\left(\frac{z}{\varepsilon}\right)  \partial^2_{ij} \chi_r (z)\right) \right]^{2}  \sqrt{|g_{\xi}(z)|}\ dz\\
     &\quad  =  \ep^{4-n}\int_{B(0,r)}  \left[\frac{1}{\sqrt{|g_{\xi}(z)|}} \sum_{i,j=1}^{n}  \partial_{i} \left( g_{\xi}^{i j}(z) \sqrt{|g_{\xi}(z)|} \right)
     \left( \frac{1}{\varepsilon} \partial_{j}  U\left(\frac{z}{\varepsilon}\right) \chi_r(z)\right)\right. \\
    &\quad \left.+ \sum_{i,j=1}^{n} g_{\xi}^{i j}(z) \,\, \left( \frac{1}{\varepsilon^2} \partial^2_{ij}  U\left(\frac{z}{\varepsilon}\right) \chi_r(z)\right) \right]^{2}  \sqrt{|g_{\xi}(z)|}\ dz+ o(\ep^{2}),\\
\end{align*}
\noindent
where the error term arises from disregarding the derivatives of $\chi_{r}$. By implementing a variable transformation, we obtain
\begin{align*}
&\ep^{4-n}\int_{M} |\Delta_{g} W_{\varepsilon,\xi}(x)|^{2} \  dV_{g}(x) \\
           &  = \ep^{4-n}\int_{B(0,r)} \left[ \sum_{i,j=1}^{n}\frac{1}{\sqrt{|g_{\xi}(z)|}}  \partial_{i} \left( g_{\xi}^{i j}(z) \sqrt{|g_{\xi}(z)|} \right)
          \left( \frac{1}{\varepsilon} \partial_{j}  U\left(\frac{z}{\varepsilon}\right) \chi_r(z)\right)\right]^2 \sqrt{|g_{\xi}(z)|} \ dz\\
            &\quad   +2 \ep^{4-n}\int_{B(0,r)}  \left[\sum_{i,j=1}^{n} \frac{1}{\sqrt{|g_{\xi}(z)|}}  \partial_{i} \left( g_{\xi}^{i j}(z) \sqrt{|g_{\xi}(z)|} \right)\left( \frac{1}{\varepsilon} \partial_{j}  U\left(\frac{z}{\varepsilon}\right) \chi_r(z)\right) \right]\times\\
          &
        \quad\left[\sum_{i,j=1}^{n}\frac{1}{\varepsilon^2} g_{\xi}^{i j}(z) \,\,    \partial^2_{ij}  U\left(\frac{z}{\varepsilon}\right) \chi_r(z)  \right] \sqrt{|g_{\xi}(z)|} \ dz \\
          &\quad 
          +\ep^{4-n}\int_{B(0,r)} \left[ \sum_{i,j=1}^{n}\frac{1}{\varepsilon^2}  \,g_{\xi}^{i j}(z)  \partial^2_{ij}  U\left(\frac{z}{\varepsilon}\right) \chi_r(z) \right]^{2}\,  \sqrt{|g_{\xi}(z)|}\ dz+ o(\ep^{2})
           \end{align*}
            \begin{align*}
          &= \int_{B(0,r/\ep)} \left( \sum_{i,j=1}^{n} g_{\xi}^{i j}( \ep z) \partial^{2}_{i j} U(z) \chi_{r}(\ep z)  \right)^{2} \sqrt{|g_{\xi}(\ep z)|} \ dz+ o(\ep^{2}).
          \end{align*}
          
    Here, the final term is the sole term with an order greater than $\ep^2$. Consequently,  from the Taylor expansions in Lemma \ref{lema:Taylor} we obtain
          \begin{align*}
          &\ep^{4-n}\int_{M} |\Delta_{g} W_{\varepsilon,\xi}(x)|^{2} \  dV_{g}(x) \\
          &= \int_{B(0,r/\ep)} \left(  \Delta U(z) + \frac{\ep^{2}}{2} \sum_{i,j,r,k=1}^{n} \frac{\partial^{2} g_{\xi}^{i j} }{\partial z_{r} \partial z_{k}} (0) z_{r} z_{k} \partial_{ij}^{2} U(z) \right)^{2} \chi_{r}^{2}(\ep z) \\
          &\quad\times \left( 1 -  \frac{\ep^{2}}{4} \sum_{l,r,k=1}^{n} \frac{\partial^{2} g_{\xi}^{l l} }{\partial z_{r} \partial z_{k}} (0) z_{r} z_{k}  \right) \ dz+ o(\ep^{2}) \\
          &= \int_{\RR^{n}} \left(  \Delta U(z) \right)^{2}\ dz+\ep^{2}  \int_{\RR^{n}}  \Delta U(z)  \left( \sum_{i,j,r,k=1}^{n}  \frac{\partial^{2} g_{\xi}^{i j} }{\partial z_{r} \partial z_{k}} (0) z_{r} z_{k} \partial_{ij}^{2} U(z)\right. \\
          &\quad- \left.\frac{1}{4} \sum_{l,r,k=1}^{n} \frac{\partial^{2} g_{\xi}^{l l} }{\partial z_{r} \partial z_{k}} (0) z_{r} z_{k}\Delta U(z) \right) \ dz+ o(\ep^{2}).
\end{align*}

To deal with the remaining terms, we proceed as in \cite[Lemma 5.3]{micheletti2009role}, and we have
\[
\begin{aligned}
&\ep^{2-n}\frac{b}{2}\int_{M}\left|\nabla  W_{\varepsilon, \xi}\right|^{2} \ dV_g
 =\frac{b}{2} \int_{\mathbb{R}^{n}}|\nabla U|^{2} \ d z+\frac{\varepsilon^{2}b}{4} \sum_{i, j, h, k=1}^{n} \frac{\partial^{2} g_{\xi}^{i j}}{\partial z_{h} \partial z_{k}}(0) \int_{\mathbb{R}^{n}}\left(\frac{U^{\prime}(|z|)}{|z|}\right)^{2} z_{i} z_{j} z_{h} z_{k} \ dz\\
& \quad-\frac{\varepsilon^{2}b}{8} \sum_{i, h=1}^{n} \frac{\partial^{2} g_{\xi}^{i i}}{\partial z_{h}^{2}}(0) \int_{\mathbb{R}^{n}}|\nabla U|^{2} z_{h}^{2}\  d z+o\left(\varepsilon^{2}\right), \qquad \text{and} \qquad\\
&\frac{1}{p+1}\int_{M} W_{\varepsilon, \xi}^{p+1} \ dV_g=
\frac{1}{p+1} \int_{\mathbb{R}^{n}} U^{p+1} \ d z-\frac{\varepsilon^{2}}{4 (p+1)} \sum_{i, j=1}^{n} \frac{\partial^{2} g_{\xi}^{i i}}{\partial z_{j}^{2}}(0) \int_{\mathbb{R}^{n}} U^{p+1} z_{j}^{2} \ d z+o\left(\varepsilon^{2}\right) .
\end{aligned}
\]
Finally, we get
\begin{equation}\label{J-alfa}
\begin{aligned}
    &J_\ep(W_{\ep,\xi})-\int_{\RR^{n}} \left(\frac{1}{2}\left|  \Delta U(z) \right|^{2}+\frac{b}{2} |\nabla U|^{2}+\frac{a}{2} U^{2} -\frac{1}{p+1}U^{p+1}\right)\ dz\\ 
    &=\ep^{2} \left[ \frac{1}{2}\sum_{i, j, h, k=1}^{n} \frac{\partial^{2} g_{\xi}^{i j}}{\partial z_{h} \partial z_{k}}(0) \int_{\RR^{n}} \partial_{ij}^{2} U(z)\Delta U(z) z_hz_k \ dz
    - \frac{1}{4} \sum_{l,h,k} \frac{\partial^{2} g_{\xi}^{l l}}{\partial z_{h} \partial z_{k}}(0) \int_{\RR^{n}}\left(\Delta U(z)\right)^2  z_hz_k  \ dz\right.\\
    &+\frac{b}{4} \sum_{i, j, h, k=1}^{n} \frac{\partial^{2} g_{\xi}^{i j}}{\partial z_{h} \partial z_{k}}(0) \int_{\mathbb{R}^{n}}\left(\frac{U^{\prime}(|z|)}{|z|}\right)^{2} z_{i} z_{j} z_{h} z_{k} \ dz\\
    &-\left.\sum_{l, j=1}^{n} \frac{\partial^{2} g_{\xi}^{ll}}{\partial z_{h}^{2}}(0)\int_{\mathbb{R}^{n}}\left(\frac{b}{8} |\nabla U|^{2} -\frac{1}{4 (p+1)}  U^{p+1} -\frac{a}{8} U^{2}  \right) z_{h}^{2}\ dz\right] + o(\ep^{2}).
\end{aligned}
\end{equation}
  
 As $U$ is a radial function, we have 
   \[
   \sum_{l,h,k} \left(\frac{\partial^{2} g_{\xi}^{l l}}{\partial z_{h} \partial z_{k}}(0) \int_{\RR^{n}}\left(\Delta U(z)\right)^2  z_hz_k  \ dz \right)
   = \left(\sum_{i,k} \frac{\partial^{2} g_{\xi}^{i i}}{ \partial z_{k}^2}(0)\right) \int_{\RR^{n}}\left(\Delta U(z)\right)^2  z_1^2  \ dz.
\]
On the other side, for each $k=1, \dots, n$, there holds
   \[
   \int_{\mathbb{R}^{n}}\left(\frac{b}{8} |\nabla U|^{2} -\frac{1}{4 (p+1)}  U^{p+1} -\frac{a}{8} U^{2}  \right) z_{h}^{2}\ dz
   =\int_{\mathbb{R}^{n}}\left(\frac{b}{8} |\nabla U|^{2} -\frac{1}{4 (p+1)}  U^{p+1} -\frac{a}{8} U^{2}  \right) z_{1}^{2}\ dz.
\]
From there, it follows that the second and last terms in the r.h.s. of \eqref{J-alfa} cancel each other out.  To compute the remaining terms in the equation \eqref{J-alfa} we use the auxiliary Lemmas \ref{LemmaAux1} and \ref{LemmaAux2} and we get
\[
    J_\ep(W_{\ep,\xi})-\int_{\RR^{n}} \left(\frac{1}{2}\left|  \Delta U(z) \right|^{2}+\frac{b}{2} |\nabla U|^{2}+\frac{a}{2} U^{2} -\frac{1}{p+1}U^{p+1}\right)\ dz
    = \ \beta \varepsilon^2 \ \tau_g(\xi) \ + \ o(\varepsilon^{2})
\]
and the result follows.
\end{proof}

In the following, we will present two auxiliary Lemmas that we have used in the previous proof. We follow the ideas of \cite{ambrosetti1999multiplicity}.

\begin{lemma}\label{LemmaAux1} There holds
    \[
    \begin{aligned}
    &\sum_{i, j, h, k=1}^{n} \left(\frac{\partial^{2} g_{\xi}^{i j}}{\partial z_{h} \partial z_{k}}(0) \int_{\RR^{n}} \partial_{ij}^{2} U(z)\Delta U(z) z_h z_k \ dz\right)\\
    &\quad=\left(\sum_{i, k} \frac{\partial^2 g_\xi^{ii}(0)}{\partial z_k^2}\right) \int_{\RR^{n}}  \left(-\frac{U'(|z|)}{|z|^3}+\frac{U''(|z|)}{|z|^2}\right)\Delta U(z) z_{1}^{2} z_{2}^{2}\ dz  
    \end{aligned}
    \]
    
\end{lemma}

\begin{proof}
Note that, as $U$ is a radial function, it satisfies
$$\int_{\RR^{n}} \partial_{ij}^{2} U(z)\Delta U(z) z_h z_k \ dz= \int_{\RR^{n}} \left(-\frac{U'(|z|)}{|z|^3}+\frac{U''(|z|)}{|z|^2}\right)\Delta U(z) z_i z_j z_h z_k \ dz.$$
Let us denote $I(p(z))=\displaystyle \int_{\RR^{n}} \left(-\frac{U'(|z|)}{|z|^3}+\frac{U''(|z|)}{|z|^2}\right)\Delta U(z) p(z) \ dz$.

Let us turn our attention to the term $I \left(z_{i} z_{j} z_{k} z_{l}\right)$ : it is different from zero only when $i=j$ and $l=k$, or when $i=k$ and $l=j$, or when $i=l$ and $j=k$. Hence, there holds
$$
\begin{aligned}
\sum_{i, j, k, l} \frac{\partial^2 g_\xi^{ij}(0)}{\partial z_k \partial z_l}& I \left(z_{i} z_{j} z_{k} z_{l}\right) \\
= & \sum_{i=j, k=l} \frac{\partial^2 g_\xi^{ij}(0)}{\partial z_k \partial z_l} I \left(z_{i} z_{j} z_{k} z_{l}\right)+\sum_{i=k, j=l} \frac{\partial^2 g_\xi^{ij}(0)}{\partial z_k \partial z_l} I \left(z_{i} z_{j} z_{k} z_{l}\right) \\
& +\sum_{i=l, j=k} \frac{\partial^2 g_\xi^{ij}(0)}{\partial z_k \partial z_l} I \left(z_{i} z_{j} z_{k} z_{l}\right)-2 \sum_{i=j=k=l} \frac{\partial^2 g_\xi^{ij}(0)}{\partial z_k \partial z_l} I \left(z_{i} z_{j} z_{k} z_{l}\right) \\
= & \sum_{i, k} \frac{\partial^2 g_\xi^{ii}(0)}{\partial z_k^2} I \left(z_{i}^{2} z_{k}^{2}\right)+2 \sum_{i, j}\frac{\partial^2 g_\xi^{ij}(0)}{\partial z_i\partial z_j} I \left(z_{i}^{2} z_{j}^{2}\right) -2 \sum_{i}\frac{\partial^2 g_\xi^{ii}(0)}{\partial z_i^2} I \left(z_{i}^{4}\right) .
\end{aligned}
 $$
Since $I \left(z_{k}^{4}\right)=3 I \left(z_{k}^{2} z_{l}^{2}\right), k \neq l$, then
$$
\sum_{i, k} \frac{\partial^2 g_\xi^{ii}(0)}{\partial z_k^2} I \left(z_{i}^{2} z_{k}^{2}\right)=\sum_{i, k} \frac{\partial^2 g_\xi^{ii}(0)}{\partial z_k^2} I \left(z_{1}^{2} z_{2}^{2}\right)+\frac{2}{3} \sum_{i}\frac{\partial^2 g_\xi^{ii}(0)}{\partial z_i^2} I \left(z_{i}^{4}\right)
$$
and
$$
\sum_{i, j}\frac{\partial^2 g_\xi^{ij}(0)}{\partial z_i\partial z_j} I \left(z_{i}^{2} z_{j}^{2}\right)=\sum_{i, j} \frac{\partial^2 g_\xi^{ij}(0)}{\partial z_i \partial z_j} I \left(z_{1}^{2} z_{2}^{2}\right)+\frac{2}{3} \sum_{i}\frac{\partial^2 g_\xi^{ii}(0)}{\partial z_i^2} I \left(z_{i}^{4}\right) .
$$
The last two equalities and equation \eqref{def:tau} imply
$$
 \sum_{i, j, k, l} \frac{\partial^2 g_\xi^{ij}(0)}{\partial z_k \partial z_l} I \left(z_{i} z_{j} z_{k} z_{l}\right) 
\ =\sum_{i, j} \frac{\partial^2 g_\xi^{ii}(0)}{\partial z_j^2} I \left(z_{1}^{2} z_{2}^{2}\right)+2 \sum_{i, j}\frac{\partial^2 g_\xi^{ij}(0)}{\partial z_i\partial z_j} I \left(z_{1}^{2} z_{2}^{2}\right)
\ =\tau_g (\xi) I \left(z_{1}^{2} z_{2}^{2}\right).
$$
So we get the desired identity.
\end{proof}

Using the same techniques, we get the following result. 

\begin{lemma}\label{LemmaAux2} There holds
   \[
   \sum_{i, j, h, k=1}^{n} \left(\frac{\partial^{2} g_{\xi}^{i j}}{\partial z_{h} \partial z_{k}}(0) \int_{\mathbb{R}^{n}}\left(\frac{U^{\prime}(|z|)}{|z|}\right)^{2} z_{i} z_{j} z_{h} z_{k} \ dz\right) =\left(\sum_{i, j=1}^{n} \frac{\partial^{2} g_{\xi}^{i i}}{\partial z_{k}^2}(0) \right)\int_{\mathbb{R}^{n}}\left(\frac{U^{\prime}(|z|)}{|z|}\right)^{2} z_{1}^2 z_{2}^2 \ dz.
\]
    
\end{lemma}

\begin{proof}
    Denote $I(p(z))=\displaystyle \int_{\RR^{n}} \left(\frac{U^{\prime}(|z|)}{|z|}\right)^{2} p(z) \ dz,$ and follow the same steps in the proof of the previous lemma.
\end{proof}

\section{The reduced problem}\label{sec:Reduced}

 In this section, we study the problem \eqref{PiP}. For $\xi\in M$  we consider the unique
$\phi_{\varepsilon,\xi} \in K_{\varepsilon,\xi}^{\perp}$ given by Proposition \ref{prop1} that solves problem \eqref{Pii} and define the function $\overline{J}_{\varepsilon}:M\to\RR$, by
$$
\overline{J}_{\varepsilon}(\displaystyle \xi)= J_{\varepsilon} (W_{\varepsilon,\xi}+\phi_{\varepsilon,\xi}).
$$

\begin{lemma} \label{lema:Joverline} It holds
	\begin{equation*}
	\overline{J_{\ep}}(\xi )=J_{\ep}(W_{\ep,\xi }+\phi_{\ep,\xi })=J_{\ep}(W_{\ep,\xi })+o(\ep^{2})
	\end{equation*}
	$C^{0}$-uniformly  in compact sets of $M$.
\end{lemma}

\begin{proof}
Since $\phi_{\ep ,\xi } \in
	K_{\ep,\xi }^{\perp}$ and satisfies \eqref{Pii}, we have
	\begin{align*}
	0 
	&=\left\langle \phi_{\ep ,\xi } ,  W_{\ep,\xi } + \phi_{\ep ,\xi }  - \io_{\ep}^{*} (f
	\left(W_{\ep,\xi } + \phi_{\ep ,\xi } ) \right) \right\rangle_{\ep} \\
 &=\Vert\phi_{\ep,\xi }\Vert_{\ep}^{2}+  \left\langle \phi_{\ep,\xi}, W_{\ep,\xi} \right\rangle_\ep  -\frac{1}{\ep^{n}}\int_{M}f(W_{\ep,\xi }+\phi_{\ep,\xi })\phi_{\ep,\xi }\ dV_{g}(x) 
	\end{align*}
	Therefore, if we set  $F(u)=\frac{1}{p+1}(u^+)^{p+1}$, then we get 
	\begin{align*}
	J_{\ep}(W_{\ep,\xi }+\phi_{\ep,\xi })-J_{\ep}(W_{\ep,\xi })&=\frac{1}{2}\Vert\phi_{\ep,\xi }\Vert_{\ep}^{2}+  \left\langle \phi_{\ep,\xi}, W_{\ep,\xi} \right\rangle_\ep 
	-\frac{1}{\ep^{n}}\int_M f(W_{\ep,\xi })\phi_{\ep,\xi } \ dV_{g}(x) \\
	&-\frac{1}{\ep^{n}}\int_{M}\left[F(W_{\ep,\xi }+\phi_{\ep,\xi })
	-F(W_{\ep,\xi })-f(W_{\ep,\xi })\phi_{\ep,\xi }\right] \ dV_{g}(x)\\
	&=-\frac{1}{2}\Vert\phi_{\ep,\xi }\Vert_{\ep}^{2}+\frac{1}{\ep^{n}}\int_{M}
	\left[f(W_{\ep,\xi }+\phi_{\ep,\xi })-f(W_{\ep,\xi })\right]\phi_{\ep,\xi } \ dV_{g}(x)\\
	&- \displaystyle \frac{1}{\ep^{n}}\int_{M}\left[F(W_{\ep,\xi }+\phi_{\ep,\xi })
	-F(W_{\ep,\xi })-f(W_{\ep,\xi })\phi_{\ep,\xi }\right] \ dV_{g}(x).
	\end{align*}
	By the Mean Value Theorem 
	we know that for some $t_{1}, t_{2}\in[0,1]$ it holds
	\[
	\displaystyle \frac{1}{\ep^{n}}\int_{M}\left[f(W_{\ep,\xi }+\phi_{\ep,\xi })-
	f(W_{\ep,\xi })\right]\phi_{\ep,\xi } \ dV_{g}(x)
        =\frac{1}{\ep^{n}}\int_{M}f'(W_{\ep,\xi }+t_{1}
	\phi_{\ep,\xi })\phi_{\ep,\xi}^{2} \ dV_{g}(x),
	\]
	and
	\[
	\frac{1}{\ep^{n}}\int_{M}\left[F(W_{\ep,\xi }+\phi_{\ep,\xi })-F(W_{\ep,\xi })-
	f(W_{\ep,\xi })\phi_{\ep,\xi }\right] \ dV_{g}(x)
	=\displaystyle \frac{1}{2\ep^{n}}\int_{M}f'(W_{\ep,\xi }+t_{2}\phi_{\ep,\xi })\phi_{\ep,\xi }^{2} \ dV_{g}(x).
	\]
	Moreover we have for any $t\in[0, 1]$
	\[
	\frac{1}{\ep^{n}}\int_M|f'(W_{\ep,\xi }+t\phi_{\ep,\xi })|\phi_{\ep,\xi }^{2} \ dV_{g}(x)\leq c\frac{1}{\ep^{n}}
	\int_M W_{\ep,\xi }^{p-1}\phi_{\ep,\xi }^{2} \ dV_{g}(x)+c\frac{1}{\ep^{n}}\int_M \phi_{\ep,\xi }^{p+1} \ dV_{g}(x)
	\]
	\[
	\displaystyle \leq c\frac{1}{\ep^{n}}\int_M\phi_{\ep,\xi }^{2} \ dV_{g}(x)+c\frac{1}{\ep^{n}}\int_M\phi_{\ep,\xi }^{p+1} \ dV_{g}(x)
	\leq C(\Vert\phi_{\ep,\xi }\Vert_{\ep}^{2}+\Vert\phi_{\ep,\xi }\Vert_{\ep}^{p+1})=o(\ep^{2}).
	\]
	In the last inequality we use  (\ref{norma_q}) and by Proposition \ref{prop1} we get $\Vert\phi_{\ep,\xi }\Vert_{\ep} = o(\ep^2 )$. 
\end{proof}

\smallskip

It is well known that any critical point of  $J_\ep$ is solution to problem \eqref{MainEq2}.
Now, we prove the following:

\begin{proposition}\label{prop3}  
    For any $\xi\in M$ we have
	\begin{equation}
	\label{expansion}
	\overline{J}_{\varepsilon}(\displaystyle \xi)=\alpha + \beta\ \varepsilon^{2}\tau_{g}(\displaystyle \xi) + o(\varepsilon^{2}), 
	\end{equation}
	$C^{0}$-uniformly with respect to $\xi$ in compact sets of $M$ as $\varepsilon\to 0$, 
	with $\alpha$ and $\beta$ given in \eqref{alfa} and \eqref{beta}, and the function $\tau_g$ is defined in \eqref{tau}. Moreover, if $\xi_{\varepsilon}$ is a critical point of $\overline{J_{\varepsilon}}$, then the function
	$W_{\varepsilon,\xi_{\varepsilon}}+\phi_{\varepsilon,\xi_{\varepsilon}}$  is a solution to the problem $(\ref{PiP}).$
\end{proposition}

\begin{proof}
	Proposition \ref{prop2} and Lemma \ref{lema:Joverline} prove (\ref{expansion}).
	We are left to prove that if  $\xi_{\ep }$ is a critical point of $\overline{J_{\ep }}$, 
	then the function
	$W_{\ep ,\xi_{\ep }}+\phi_{\ep ,\xi_{\ep }}$ is a
	solution to problem  \eqref{MainEq2}. 
	For $y\in B(0, r)\subset \RR^n$ we  let $\xi(y)=  \exp_{\xi_\ep}(y)\in B_g(\xi,r)\subset M$. We remark that $\xi(0)= \xi_\ep$.
	Since $\xi_\ep$ is a critical point of $\overline{J_{\ep }}$,
	\begin{equation}
	\label{A}
	\displaystyle \frac{\partial}{\partial y_{i} }\overline{J_{\ep }}(\xi(y))\Big|_{y=0}=0, \hbox{ for } \ \  i=1, . . . , n.  
	\end{equation}
	
	Note that Equation \eqref{MainEq2} is equivalent to $\nabla J_{\ep}  (u) =0$, where $\nabla J_{\ep}:H^2_\ep\to H^2_\ep$. Moreover, we can write 
	\[
	\displaystyle \nabla J_{\ep}  \left(W_{\ep ,\xi(y)}+\phi_{\ep ,\xi(y)}\right) = 
	\Pi_{\ep  , y}^{\perp} \nabla J_{\ep}  \left(W_{\ep ,\xi(y)}+\phi_{\ep ,\xi(y)}\right) +
	\Pi_{\ep  , y} \nabla J_{\ep}  \left(W_{\ep ,\xi(y)}+\phi_{\ep ,\xi(y)}\right),
	\]
	
	\noindent
	where the first term on the right is $0$ by the construction of $\phi_{\ep ,\xi(y)}$. Then, we write the second term as
	
	\[
	\Pi_{\ep  , y} \nabla J_{\ep}  \left(W_{\ep ,\xi(y)}+\phi_{\ep ,\xi(y)}\right)
	=\sum_{k} C_{\ep }^{k} Z_{\ep  ,y}^k, 
	\]
	
	\noindent 
	for some functions $C_{\ep }^{k} : B(0,r)  \rightarrow \RR$. We have to prove that for each $k=1,\dots, n$ and $\ep  >0 $ small, $C_{\ep }^{k} (0)=0$.
Indeed, if $i$ is fixed, then we have 
	\[ 0=
	\displaystyle \frac{\partial}{\partial y_{i} }\overline{J_{\ep }}(\xi(y))\Big|_{y=0}=
	\left\langle\nabla J_{\ep }\left(W_{\ep ,\xi_\ep}+\phi_{\ep ,\xi_\ep}\right) ,
	\frac{\partial}{\partial y_{i} } \Big|_{y=0} \left(W_{\ep ,\xi(y)}+\phi_{\ep ,\xi(y)}\right)\right\rangle_{\ep }
	\]
	\begin{equation}\label{B}
	=\left\langle \sum_{k}C_{\ep }^{k} (0) Z_{\ep ,\xi_\ep}^{k}, 
	\displaystyle \frac{\partial}{\partial y_{i} } \Big|_{y=0} \left(W_{\ep ,\xi(y)}+\phi_{\ep ,\xi(y)}\right)\right\rangle_{\ep }.
	\end{equation}

	Since $\phi_{\ep ,\xi(y)}\in K_{\ep ,\xi(y)}^{\perp}$, for any $k$, we have that
	$\left\langle Z_{\ep ,y }^{k}, \phi_{\ep ,\xi(y)}\right\rangle_{\ep }=0$. Then
\[
 \left\langle Z_{\ep ,\xi_\ep }^{k},\ \frac{\partial}{\partial y_{i} }\Big|_{y=0}\phi_{\ep ,\xi(y)}\right\rangle_{\ep } + \left\langle  \frac{\partial}{\partial y_{i} }\Big|_{y=0} \ Z_{\ep ,y }^{k},\phi_{\ep ,\xi_\ep}\right\rangle_{\ep } =0.
\]
 It follows that
	\[ \liminf_{\ep  \rightarrow 0}
	\left\lvert  \left\langle Z_{\ep ,\xi_\ep }^{k},\ 
 \frac{\partial}{\partial y_{i}}\Big|_{y=0}\phi_{\ep ,\xi(y)}\right\rangle_{\ep }   \right\rvert
	= \liminf_{\ep  \rightarrow 0} \left\lvert
	-\left\langle\frac{\partial}{\partial y_{i} }\Big|_{y=0} Z_{\ep , \xi(y) }^{k},\ \phi_{\ep,\xi_\ep  }\right\rangle_{\ep } \right\rvert
	\]
	\[
	\leq  \liminf_{\ep  \rightarrow 0} \displaystyle \left\Vert\frac{\partial}{\partial y_{i} }\Big|_{y=0} Z_{\ep , \xi(y)  }^{k}
	\right\Vert_{\ep } \ \left\Vert\phi_{\ep ,\xi_\ep } \right\Vert_{\ep }=0,
	\]
	 where the last equality follows from Proposition \ref{prop1} and (\ref{DZ}). On the other side, we have
%
 from (\ref{AA}) that
	
	\[
	\lim_{\ep  \rightarrow 0} \ \ep  \
	\left\langle \sum_{k}C_{\ep }^{k} (0) Z_{\ep ,\xi_\ep }^{k}, 
	\displaystyle \frac{\partial}{\partial y_{i} } \Big|_{y=0} W_{\ep ,\xi(y)}\right\rangle_{\ep }=C_{\ep }^{i} (0)  C.
	\]
	And then it follows from (\ref{B}) that $C_{\ep }^{i} (0) =0$ for all $i=1, \dots, n$.
\end{proof}
 
\smallskip

Now, we are ready to prove our main result. 

\begin{proof}[Proof of Theorem \ref{mainthm}] 
%



Let us assume that $\xi_0$ is an isolated maximum for the function $\tau_g$. Then, there is $\rho>0$ such that $\xi_0$ is the only maximum point of $\tau_g$ in $B_g(\xi_0, \rho)$.
 As $\overline{J}_\ep\in C^0(M)$, we know that there is a $ \xi_{\ep } \in \overline{B_g(\xi_0, \rho) }$ which satisfies that
 \begin{equation}\label{J_max en xi_ep}
      \overline{J}(\xi_\ep)\geq \overline{J}(\xi), \quad \text{for all }\xi\in B_g(\xi_0,\rho).
 \end{equation}
 By Proposition \ref{prop3} we have 
 \begin{equation}\label{J bar en xi_ep}
     \overline{J}_{\varepsilon}( \xi_\ep)=\alpha + \beta\ \varepsilon^{2}\tau_{g}(\xi_\ep) + o(\varepsilon^{2}). 
 \end{equation}
Now, consider an auxiliary point, given by $\xi^*_\ep=\exp_{\xi_0}(\sqrt{\ep}z)$, for some $z\in\RR^n$ with $|z|=1$. Then $\xi^*_\ep\in B_g(\xi_0, \rho)$ for $\ep$ small enough as 
$d_g(\xi_0, \xi^*_\ep)^2\leq \ep$. Using the Taylor expansion of $\tau_g$ around $\xi_0$, we have
\[
\tau_g(\xi^*_\ep)=\tau_g(\xi_0)+ o(1)
\]
and it follows from there and Proposition \ref{prop3} that
\begin{equation}\label{auxiliarpoint}
\overline{J}_{\varepsilon}( \xi^*_\ep)=\alpha + \beta\ \varepsilon^{2}\tau_{g}(\xi_0) + o(\varepsilon^{2}). 
\end{equation}
Reading at the equations \eqref{J_max en xi_ep}, \eqref{J bar en xi_ep} and \eqref{auxiliarpoint} we deduce
\begin{equation} \label{minimo_tau}
    \beta\left(\tau_g(\xi_\ep)-\tau_g(\xi_0)\right)\geq o(1),
\end{equation}
and as $\xi_0$ is the only maximum point of $\tau_g$ in $B_g(\xi_0, \rho)$ and $\beta>0$, then we have by \eqref{J bar en xi_ep} and \eqref{minimo_tau} that
\[
\lim_{\ep\to 0} \tau_g(\xi_\ep)-\tau_g(\xi_0)=0 \quad \text{and } \quad 
\lim_{\ep\to 0} \xi_\ep-\xi_0=0.
\]
Moreover, by Proposition \ref{prop3}, the function $u_{\ep }= W_{\ep  , \xi_{\ep }}
	+\phi_{\ep  , \xi_{\ep }}$ is a solution to problem \eqref{MainEq2} and by Proposition \ref{prop1}, we have
\begin{equation}\label{concentration}
    \| u_{\ep } - W_{\ep  , \xi_{\ep }} \|_\ep
	=\| \phi_{\ep  , \xi_{\ep }}  \|_\ep = o(\ep).
\end{equation}
So, we proved that the solution $u_\ep$ to Equation \eqref{MainEq2} is concentrated around $\xi_0$ as $\ep\to0$.
\end{proof}

\smallskip

\section{Appendix A: Some computations}\label{sec:Appendix}
In this Section, we compute the estimations omitted in \eqref{R-estimate2}.

\begin{lemma} It holds
    \begin{equation*}
\int_{B(0,r)} \left|\mathcal{R}_\ep(U_{\ep},\chi_{r}) \right|^{\frac{p+1}{p}}\ dz=o\left(\ep^{n+2\frac{p+1}{p}}\right).
\end{equation*}
\end{lemma}

\begin{proof}
  Recall that in the following expression, we are using the Einstein notation. 
    \begin{align*}
    \mathcal{R}_\ep(U_{\ep},\chi_{r})&=2\ep^4\Delta U_\ep(z) \Delta\chi_{r}(z)
    +  4\ep^4 \partial^3_{ijj} U_\ep(z)\partial_i \chi_r(z)\\
&\quad+4\ep^4\partial_{ij}^2 U_\ep(z) \partial_{ij}^2\chi_r (z)
    + 4\ep^4 \partial_{i}U_\ep(z)\partial^3_{ijj} \chi_r(z) 
    +\ep^4 U_\ep(z)\Delta^2 \chi_r(z)\\
      &\quad - \ep^4\Delta \left(A^{st} \partial_{s t}^{2}  (U_{\ep} \chi_{r}) \right)(z) 
      - \ep^4\Delta \left(B^{h} \partial_{h} (U_{\ep} \chi_{r})\right)(z) \\
    &\quad - \ep^4 A^{ij} \partial_{i j}^{2} \left(\Delta(U_{\ep} \chi_{r})\right)(z) 
    +\ep^4 A^{ij} \partial_{i j}^{2} \left( A^{st} \partial_{s t}^{2}  (U_{\ep} \chi_{r}) )\right)(z) 
    + \ep^4 A^{ij} \partial_{i j}^{2} \left(B^{h} \partial_{h} (U_{\ep} \chi_{r})\right)(z) \\
    &\quad - \ep^4 B^{k} \partial_{k}\left(\Delta (U_{\ep} \chi_{r}) \right)(z)
    + \ep^4 B^{k} \partial_{k} \left(A^{st} \partial_{s t}^{2} (U_{\ep} \chi_{r})\right)(z) 
    + \ep^4 B^{k} \partial_{k}\left(B^{h}\partial_{h}\left(U_{\ep} \chi_{r}\right)\right)(z)\\
    &\quad -2b \ep^2\partial_{i} U_\ep(z) \partial_{i} \chi_{r}(z)  
    - b \ep^2 U_\ep(z) \Delta \chi_{r}(z)  
    -b\ep^2 B^{k}\partial_{k} \left(U_{\ep} \chi_{r}\right)(z)
    +b \ep^2 A^{ij} \partial^2_{ij} \left(U_{\ep} \chi_{r}\right)(z). 
\end{align*}
Considering the decay of $U$ and the boundedness of $\chi_r$, as given in \eqref{decay} and \eqref{chi-decay}, we get that
   $$
 2\varepsilon^{4 \frac{p+1}{p}} \int_{B(0, r)} \left|\Delta U_{\varepsilon}\right|^{\frac{p+1}{p}}(z)\left|\Delta \chi_{r}(z)\right|^{\frac{p+1}{p}} d z \leq c \varepsilon^{4 \frac{p+1}{p}} \int_{B(0, r) \backslash B(0, r / 2)} \left|\Delta U_{\varepsilon}\right|^{ \frac{p+1}{p}}(z) d z=o\left(\ep^{n+4 \frac{p+1}{p}}\right)
$$
  
$$
 4\ep^{4 \frac{p+1}{p}} \int_{B(0, r)} \left|\partial^3_{ijj} U_\ep(z)\partial_i \chi_r(z)\right|^{\frac{p+1}{p}} d z \leq c \varepsilon^{4 \frac{p+1}{p}} \int_{B(0, r) \backslash B(0, r / 2)} \left|\partial^3_{ijj} U_{\varepsilon}\right|^{ \frac{p+1}{p}}(z) d z=o\left(\ep^{n+4 \frac{p+1}{p}}\right)
$$
   
$$
4\ep^{4 \frac{p+1}{p}} \int_{B(0, r)} \left|\partial_{ij}^2 U_\ep(z) \partial_{ij}^2\chi_r (z)
\right|^{\frac{p+1}{p}} d z \leq c \varepsilon^{4 \frac{p+1}{p}} \int_{B(0, r) \backslash B(0, r / 2)} \left|\partial_{ij}^2 U_\ep\right|^{\frac{p+1}{p}}(z) d z=o\left(\ep^{n+4 \frac{p+1}{p}}\right)
$$
   
$$
4\ep^{4 \frac{p+1}{p}} \int_{B(0, r)} \left| \partial_{i}U_\ep(z)\partial^3_{ijj} \chi_r(z)\right|^{\frac{p+1}{p}} d z \leq c \varepsilon^{4 \frac{p+1}{p}} \int_{B(0, r) \backslash B(0, r / 2)} \left| \partial_{i}U_\ep\right|^{\frac{p+1}{p}}(z) d z=o\left(\ep^{n+4 \frac{p+1}{p}}\right)
$$
   
$$
\ep^{4 \frac{p+1}{p}} \int_{B(0, r)} \left| U_\ep(z)\Delta^2 \chi_r(z)\right|^{\frac{p+1}{p}} d z \leq c \varepsilon^{4 \frac{p+1}{p}}  \int_{B(0, r) \backslash B(0, r / 2)} \left| U_\ep\right|^{\frac{p+1}{p}}(z) d z=o\left(\ep^{n+4 \frac{p+1}{p}}\right)
$$
  
  $$
2b\varepsilon^{2 \frac{p+1}{p}} \int_{B(0, r)}\left|\nabla U_{\varepsilon}(z)\cdot \nabla \chi_{r}(z)\right|^{\frac{p+1}{p}} d z \leq c \varepsilon^{2 \frac{p+1}{p}} \int_{B(0, r) \backslash B(0, r / 2)}\left|\nabla U_{\varepsilon}\right|^{\frac{p+1}{p}}(z) d z=o\left(\ep^{n+2 \frac{p+1}{p}}\right)
$$  
and
 $$
 b \ep^{2 \frac{p+1}{p}} \int_{B(0, r)}\left| U_\ep(z) \Delta \chi_{r}(z)\right|^{\frac{p+1}{p}} d z \leq c \varepsilon^{2 \frac{p+1}{p}} \int_{B(0, r) \backslash B(0, r / 2)}\left| U_{\varepsilon}\right|^{\frac{p+1}{p}}(z) d z=o\left(\ep^{n+2 \frac{p+1}{p}}\right).
 $$  
By the standard properties of the exponential map established in Lemma \ref{lema:Taylor}, we know that there exists a positive constant $C$ such that for any point $z\in B(0,z)$ and any indices $i, j,$ and $k,$ it holds
$$
 |A^{ij}(\ep z)|=\left|g^{i j}(\ep z)-\delta^{ij}(\ep z)\right|\leq C|\ep z|^2,
\text{ and } \quad |B^{k}(\ep z)|=|g^{i j}(\ep z) \Gamma_{i j}^{k}(\ep z)|\leq C |\ep z|.
$$
Then we have 
$$
\begin{aligned}
& \varepsilon^{2 \frac{p+1}{p}} \int_{B(0, r)}\left|A^{ij}(z) \partial^2_{i j}\left(U_{\varepsilon} \chi_{r}\right)(z)\right|^{\frac{p+1}{p}} \ d z  \\
&=  \ \varepsilon^{2 \frac{p+1}{p}} \int_{B(0, r)}\left|\left(g_{\xi}^{ij}( z)-\delta_{i j}\right)  \left(\partial^2_{i j}U_{\varepsilon} \ \chi_{r} + 2 \partial_i U_\ep \ \partial_j \chi_r + U_\ep \ \partial_{ij}^2 \chi_r\right)(z)\right|^{\frac{p+1}{p}} \ d z \\
& \leq c \varepsilon^{n} \int_{B(0, r / \varepsilon)}\left|\left(g_{\xi}^{i j}(\varepsilon z)-\delta_{i j}\right) \partial_{i j}^2 U(z)\right|^{\frac{p+1}{p}}\  d z \\
& \quad+c \varepsilon^{4 \frac{p+1}{p}} \int_{B(0, r) \backslash B(0, r / 2)} \left|z^2 \ U_{\varepsilon}(z)\right|^{\frac{p+1}{p}} d z+c \varepsilon^{4 \frac{p+1}{p}} \int_{B(0, r) \backslash B(0, r / 2)}\left|z^2\left(\partial_{i} U_{\varepsilon}\right)(z)\right|^{\frac{p+1}{p}} d z \\
& =O\left(\varepsilon^{n+2 \frac{p+1}{p}}\right).
\end{aligned}
$$

Similarly, since $M$ is an analytic manifold, we obtain bounds for the derivatives of $A^{ij}$ from the estimates given in Lemma \ref{lema:Taylor} as well. More precisely, we have that

$$
\begin{aligned}
&\ep^{4 \frac{p+1}{p}} \int_{B(0, r)} \left|A^{ij} \partial_{i j}^{2} \left(\Delta(U_{\ep} \chi_{r})\right)(z)  \right|^{\frac{p+1}{p}} \ d z \\ 
&\quad = c\varepsilon^{n} \int_{B(0, r / \varepsilon)}\left|A^{ij} \partial_{i j}^2 \Delta U(z)\right|^{\frac{p+1}{p}}\  d z 
 + o\left(\ep^{n+4 \frac{p+1}{p}}\right) \leq O\left(\varepsilon^{n+2 \frac{p+1}{p}}\right).
\end{aligned}
$$
Moreover
$$
\begin{aligned}
&\ep^{4 \frac{p+1}{p}} \int_{B(0, r)} \left| \Delta \left(A^{st} \partial_{s t}^{2}  (U_{\ep} \chi_{r}) \right)(z) \right|^{\frac{p+1}{p}} \ d z\\
&\quad \leq c\varepsilon^{n} \int_{B(0, r / \varepsilon)}\left|A^{st} \Delta\partial_{st}^2 U(z)\right|^{\frac{p+1}{p}}\  d z +
\ep^{4 \frac{p+1}{p}} \int_{B(0, r)} \left| \Delta \left(A^{st} \right) \partial_{s t}^{2}  (U_{\ep} \chi_{r})(z) \right|^{\frac{p+1}{p}} \ d z\\
 &+ 2\ep^{4 \frac{p+1}{p}} \int_{B(0, r)} \left| \nabla \left(A^{st} \right) \nabla\partial_{s t}^{2}  (U_{\ep} \chi_{r})(z) \right|^{\frac{p+1}{p}} \ d z= O\left(\ep^{n+2 \frac{p+1}{p}}\right) 
\end{aligned}
$$
Analogously, as the derivatives of $B^{h}$ are bounded, we have 
$$
\begin{aligned}
&
\ep^{4 \frac{p+1}{p}} \int_{B(0, r)} \left| \Delta \left(B^{h} \partial_{h} (U_{\ep} \chi_{r})\right)(z) \right|^{\frac{p+1}{p}} d z = \\
&\ep^{n+ \frac{p+1}{p}} \int_{B(0, r/\ep)} \left| B^{h}(\ep z) \partial_{h} \Delta U(z)\right|^{\frac{p+1}{p}} d z + \ep^{4 \frac{p+1}{p}} \int_{B(0, r)} \left| \Delta \left(B^{h}(z)\right) \partial_{h} (U_{\ep} \chi_{r})(z) \right|^{\frac{p+1}{p}} d z \  + \\
&\ep^{4 \frac{p+1}{p}} \int_{B(0, r)} \left| \nabla \left(B^{h}(z)\right) \nabla\partial_{h} (U_{\ep} \chi_{r})(z)  \right|^{\frac{p+1}{p}} d z =
O\left(\ep^{n+2 \frac{p+1}{p}}\right)
\end{aligned}
$$
 and
$$
\ep^{4 \frac{p+1}{p}} \int_{B(0, r)} \left|  A^{ij} \partial_{i j}^{2} \left(B^{h} \partial_{h} (U_{\ep} \chi_{r})\right)(z)  \right|^{\frac{p+1}{p}} d z =O\left(\ep^{n+4 \frac{p+1}{p}}\right)
$$
  and
$$
\ep^{4 \frac{p+1}{p}} \int_{B(0, r)} \left|  B^{k} \partial_{k}\left(\Delta (U_{\ep} \chi_{r}) \right)(z) \right|^{\frac{p+1}{p}} d z =O\left(\ep^{n+2 \frac{p+1}{p}}\right)
$$
  and
 $$\ep^{4 \frac{p+1}{p}} \int_{B(0, r)} \left|   B^{k} \partial_{k} \left(A^{st} \partial_{s t}^{2} (U_{\ep} \chi_{r})\right)(z) \right|^{\frac{p+1}{p}} d z =O\left(\ep^{n+4 \frac{p+1}{p}}\right)
 $$ 
   and
 $$\ep^{4 \frac{p+1}{p}} \int_{B(0, r)} \left|   B^{k} \partial_{k}\left(B^{h}\partial_{h}\left(U_{\ep} \chi_{r}\right)\right)(z) \right|^{\frac{p+1}{p}} d z =O\left(\ep^{n+4 \frac{p+1}{p}}\right)
 $$
  and
$$
b\ep^{2 \frac{p+1}{p}} \int_{B(0, r)} \left|   B^{k}\partial_{k} \left(U_{\ep} \chi_{r}\right)(z) \right|^{\frac{p+1}{p}} d z =O\left(\ep^{n+2 \frac{p+1}{p}}\right).
$$
The result follows.
\end{proof}

\bibliographystyle{alpha}
\bibliography{yamabe}
\end{document}